\documentclass[10pt,twoside]{article}
\usepackage{mathrsfs}
\usepackage{amsmath}
\usepackage{amssymb}
\usepackage{fancyhdr}
\usepackage{latexsym}
\usepackage{bbding}
\usepackage{mathrsfs}
\usepackage{wasysym}
\usepackage{cite}

\usepackage{multicol,graphics}

\newtheorem{theorem}{Theorem}[section]
\newtheorem{lemma}[theorem]{Lemma}
\newtheorem{corollary}[theorem]{Corollary}

\newtheorem{definition}[theorem]{Definition}
\newtheorem{remark}[theorem]{Remark}

\numberwithin{equation}{section}
\newenvironment{proof}[1][Proof]{\noindent\textbf{#1.} }{\hfill $\Box$}
\allowdisplaybreaks

 \makeatletter
\begin{document}
\title{{Space-time derivative estimates of the Kock-Tataru solutions to the nematic liquid crystal system in Besov spaces}
\thanks{The author is supported by the Hunan Provincial Natural Science Foundation of China
(13JJ4043) and National Natural Science Foundation of China
(11326155, 11171357).} }
\author{{\small   Qiao Liu \thanks{\text{E-mail address}: liuqao2005@163.com.
}}
\\
%EndAName
%{\small  College of Mathematics and Computer Science, Key
%Laboratory of High Performance }\\
%{\small Computing and Stochastic Information Processing
%(Ministry of Education of China),}\\
{\small Department of Mathematics, Hunan Normal University, Changsha, Hunan 410081, P. R. China}\\
{ June 13, 2014}\\
}
\date{}
\maketitle

\begin{abstract}
In  recent paper \cite{DW1} (Y. Du and K. Wang, Space-time
regularity of the Kock $\&$ Tataru solutions to the liquid crystal
equations, SIAM J. Math. Anal., \textbf{45}(6), 3838--3853.), the
authors proved that the global-in-time Koch-Tataru type solution
$(u,d)$ to the $n$-dimensional incompressible nematic liquid crystal
flow with small initial data $(u_{0},d_{0})$ in $BMO^{-1}\times BMO$
has arbitrary space-time derivative estimates in the so called
Koch-Tataru space norms. The purpose of this paper is to show that
the Koch-Tataru type solution satisfies the decay estimates for any
space-time derivative involving some borderline Besov space norms.
More precisely, for the global-in-time Koch-Tataru type solution
$(u,d)$ to the nematic liquid crystal flow with initial data
$(u_{0},d_{0})\in BMO^{-1}\times BMO$ and
$\|u_{0}\|_{BMO^{-1}}+[d_{0}]_{ BMO}\leq \varepsilon$ for some small
enough $\varepsilon>0$, and for any positive integers $k$ and $m$,
one has
\begin{align*}
\|t^{\frac{k}{2}+m}(\partial^{k}_{t}\nabla^{m} u,
\partial^{k}_{t}\nabla^{m} \nabla
d)\|_{\widetilde{L}^{\infty}(\mathbb{R}_{+},\dot{B}^{-1}_{\infty,\infty})\cap
\widetilde{L}^{1}(\mathbb{R}_{+};\dot{B}^{1}_{\infty,\infty})}\leq
\varepsilon.
\end{align*}
Furthermore, we shall give that the solution admits  an unique
trajectory which is H\"{o}lder continuous with respect to space
variables.

\medskip


\end{abstract}

\section{Introduction}\label{Int}

\noindent

Liquid crystal refer to a state of matter that has properties
between those of a solid crystal and those of an isotropic liquid.
Examples of liquid crystals can be found both in nature (e.g.,
solutions of soap and detergents) and in technological applications
(e.g., electronic displays). There are many different types of
liquid crystal phases, %among which one of the most common phases is the nematic.
and the most common phase is the nematic. In a nematic phase, the
molecules do not exhibit any positional order, but they have
long-range orientational order.
 The hydrodynamic theory of the nematic liquid crystals,
due to Ericksen and Leslie, was developed during the period of 1958
through 1968 (see \cite{ER,LE}). Since then, many remarkable
developments have been made from both theoretical and applied
aspects.

 For any $n\geq 2$, the hydrodynamic flow of nematic liquid crystals
 in $\mathbb{R}^{n+1}_{+}:= \mathbb{R}^{n}\times\mathbb{R}_{+}$ is
 given by
\begin{align}
   \label{eq1.1}
&{\partial_{t}}u-\nu\Delta u +(u\cdot\nabla)u+\nabla{P}=-\lambda\nabla\cdot(\nabla d \odot\nabla d)\quad\text{ in }\mathbb{R}^{n+1}_{+},\\
%--------------------(eq1.1)------------------------------------
   \label{eq1.2}
&\partial_{t}d+(u\cdot\nabla)d=\gamma(\Delta d+|\nabla d|^{2}d)\quad\quad
  \quad\quad\quad\quad\quad\quad\text{ in }\mathbb{R}^{n+1}_{+},\\
%--------------------(eq1.2)------------------------------------
   \label{eq1.3}
&\nabla\cdot u=0,\quad|d|=1  \qquad\qquad\quad\quad
\quad\quad\quad\quad\quad\quad\quad\quad\text{ in }\mathbb{R}^{n+1}_{+},\\
%--------------------(eq1.3)------------------------------------
   \label{eq1.4}
&(u,d)|_{t=0}=(u_{0},d_{0}),\quad|d_{0}(x)|=1
\quad\quad\quad\quad\quad\quad\quad\quad \text{in }\mathbb{R}^{n},
%--------------------(eq1.4)------------------------------------
\end{align}
where $u(x,t):\mathbb{R}^{n+1}_{+}\rightarrow \mathbb{R}^{n}$ is the
unknown velocity field of the flow,
$P(x,t):\mathbb{R}^{n+1}_{+}\rightarrow \mathbb{R}$ is the scalar
pressure and $d:\mathbb{R}^{n+1}_{+}\rightarrow \mathbb{S}^{2}$, the
unit sphere in $\mathbb{R}^{3}$, is the unknown (averaged)
macroscopic/continuum molecule orientation of the nematic liquid
crystal flow, $\nabla \cdot u=0$ represents the incompressible
condition, $(u_{0},d_{0})$ is a given initial data with $\nabla
\cdot u_{0}=0$ in distribution sense, and  $\nu$, $\lambda$ and
$\gamma$ are positive numbers associated to the properties of the
material: $\nu$ is the kinematic viscosity, $\lambda$ is the
competition between kinetic energy and potential energy, and
$\gamma$ is the microscopic elastic relaxation time for the
molecular orientation field. The notation $\nabla d\odot\nabla
d=\left(\partial_{i}d\cdot
\partial_{j}d\right)_{1\leq i,j\leq n}$ is the stress tensor induced
by the director field $d$.  Since the concrete values of $\nu$,
$\lambda$ and $\gamma$ do not play a special role in our discussion,
for simplicity, we assume that they all equal to one throughout this
paper.

System \eqref{eq1.1}--\eqref{eq1.4} is a simplified version of the
Ericksen-Leslie model \cite{ER,LE}, but it still retains most
important  mathematical structures as well as most of the essential
difficulties of the original Ericksen-Leslie model. A brief account
of the Ericksen-Leslie theory on nematic liquid crystal flows and
the derivations of several approximate systems can be found in the
appendix of \cite{LL1}. For more details on the dynamic continuum
theory of liquid crystals, we refer the readers to the book of
Stewart \cite{IWS}. Mathematically, system
\eqref{eq1.1}--\eqref{eq1.4} is a strongly coupled system between
the incompressible Navier-Stokes (NS) equations (the case $d\equiv
1$, see e.g.,
\cite{C1,C2,M,DD,DL,D,FK,GPS,YO,YG,TK,KF,HD,PG,Leray,HO,OS,ZZ}) and
the transported heat flows of harmonic map (the case $u\equiv 0$,
see e.g., \cite{MS,W}), and thus, its mathematical analysis is full
of challenges.

To make a clearer introduction to the results of the present paper,
we shall recall some well-posedness and regularity results of the NS
equations. In the seminal paper \cite{Leray},  Leray proved the
global existence of finite energy weak solutions to
\eqref{eq1.1}--\eqref{eq1.3}, but its regularity and uniqueness
still remain open. The theory of the so called \textit{mild}
solutions to the NS equations is pioneered by Fujita and Kato
\cite{FK,KF}, and these works inspired extensive study in the
following years on the well-posedness of the NS equations in various
critical spaces, see Kato \cite{TK}, Cannone \cite{M}, Koch and
Tataru \cite{HD}, Lemari\'{e}-Rieusset \cite{PG} and so on.
Particularly, Koch and Tataru established the well-posedness for the
NS equations  with initial data in $BMO^{-1}(\mathbb{R}^{n})$.
Hereafter, we call the \textit{mild} solution presented by Koch and
Tataru \cite{HD} as Koch-Tataru solution. For the spatial regularity
on the \textit{mild} solution to the NS equations has been studied
by many authors, such as  Giga and Sawada \cite{YO}, Sawada
\cite{OS} and Miura and Sawada \cite{HO}. In paper \cite{GPS},
Germain, Pavlovi\'{c} and Staffilani had proved the Koch-Tataru's
solution $u$ satisfies the following spatial regularity property:
\begin{align*}
t^{\frac{m}{2}}\nabla^{m} u\in Z_{T^{*}} \text{ for all }
m\in\mathbb{N},
\end{align*}
where $Z_{T^{*}}$ is the Koch-Tataru's solution existence space for
the NS equations. In \cite{ZZ}, the authors generalized
Germain-Pavlovi\'{c}-Staffilani's results, and obtained that for
sufficient small enough $u_{0}\in BMO^{-1}$, the global-in-time
Koch-Tataru's solution satisfies
\begin{align*}
\|t^{\frac{m}{2}}\nabla^{m}u\|_{\widetilde{L}^{\infty}(\mathbb{R}_{+};\dot{B}^{-1}_{\infty,\infty})\cap
\widetilde{L}^{1}(\mathbb{R}_{+};\dot{B}^{1}_{\infty,\infty})}\leq
C\|u_{0}\|_{BMO^{-1}}(1+\|u_{0}\|_{BMO^{-1}})\text{ for all }
m\in\mathbb{N}.
\end{align*}
As to the space-time regularity of the \textit{mild} solutions to
the NS equations, when $u_{0}\in L^{n}(\mathbb{R}^{n})$, in
\cite{DD}, Dong and Du established the result
\begin{align*}
\|t^{\frac{m}{2}+k}\partial_{t}^{k}\nabla^{m}
u\|_{L^{n+2}(\mathbb{R}^{n}\times (0,T^{*}))}\leq C \text{ for all }
k,m\in \mathbb{N},
\end{align*}
where $T^{*}$ is the maximum existence time. Vary recently, inspired
by the results of \cite{DD} and \cite{GPS}, Du in \cite{D} proved
that the Koch-Tataru's solution is space-time regularity. More
precisely, the author established that there holds
\begin{align*}
t^{\frac{m}{2}+k}\partial_{t}^{k}\nabla^{m} u\in Z_{T^{*}} \text{
for all } k, m\in\mathbb{N},
\end{align*}
with the initial data $u_{0}\in BMO^{-1}$. On the other hand, with
suitable regularities for the solution $u$ to the NS equations,
Chemin \cite{C1,C2} proved that the existence and uniqueness of the
trajectory to $u$, moreover, this trajectory is H\"{o}lder
continuous with respect to the space variables.

 During the past several decades, there have been many attempts on
 rigorous mathematical analysis on the nematic liquid crystal flows, see, for example,
\cite{DW1,DW2,JHW,HMC,HW,HW1,XLW,L,LLW,LL1,LL2,LW,LD,LNW,LQ,SL,W,WD,XZ}
and the references therein. If $|\nabla d|^{2} d$ in \eqref{eq1.2}
is replaced by $\frac{(1-|d|^{2})d}{\varepsilon}$ ($\varepsilon$ is
a positive parameter), thus  the Dirichlet energy
\begin{align*}
\frac{1}{2}\int_{\mathbb{R}^{n}}|\nabla d|^{2}\text{d}x\quad\text{
for  } d:\mathbb{R}^{n}\rightarrow \mathbb{S}^{2}
\end{align*}
 is replaced by the Ginzburg-Landau energy
\begin{align*}
\int_{\mathbb{R}^{n}}\left(\frac{1}{2}|\nabla
d|^{2}+\frac{(1-|d|^{2})^{2}}{4\varepsilon^{2}}\right)\text{d}x
\quad \text{ for } d:\mathbb{R}^{n}\rightarrow \mathbb{R}^{n},
\varepsilon>0.
\end{align*}
In this case, the system has been studied by a series of papers by
Lin \cite{L} and Lin and Liu \cite{LL1,LL2}. More precisely, they
proved in \cite{LL1} the local classical solutions and the global
existence of weak solutions in dimensions two and three, and for any
fixed $\varepsilon$, they also obtained the existence and uniqueness
of global classical solution either in dimension two or dimension
three for large fluid viscosity $\nu$. However, as the authors
pointed out in \cite{LL1}, it is a challenging problem to study the
limiting case as $\varepsilon$ tends to zero. Later, in \cite{LL2},
they proved partial regularity of weak solutions in dimension three.
Hu and Wang \cite{HW} established the existence of global strong
solution in suitable regular space and proved that all weak
solutions constructed in \cite{LL1} must be equal to the unique
strong solution. Compared with these results, the studies for system
\eqref{eq1.1}--\eqref{eq1.4} were only started in recent years. In
papers Lin et al. \cite{LLW} and Hong \cite{HMC}, the authors proved
that there exists global Leray-Hopf type weak solutions to
\eqref{eq1.1}-\eqref{eq1.4} with suitable boundary condition in
dimension two, and established that the solutions are smooth away
from at most finitely many singular times which is similar as that
for the heat flows of harmonic maps (see \cite{MS}).  The uniqueness
of such weak solutions was subsequently obtained by Lin and Wang
\cite{LW} and Xu and Zhang \cite{XZ}.  the existence of
local-in-time strong solution with large initial value and
global-in-time strong solution with  small initial value of system
\eqref{eq1.1}--\eqref{eq1.4} were also considered by many authors,
we refer the readers to see \cite{XLW,WD,LD,LQ,JHW} and the
references cited therein.
%   Li and Wang \cite{XLW} established the existence of local strong
%   solution with large initial value and global strong solution with
%   small initial value for the initial-boundary value problem. Wen and
%   Ding \cite{WD} obtained local existence and uniqueness of strong
%   solution when the initial data is smooth enough for the Cauchy
%   problem.  Lin and Ding \cite{LD} obtained local and global
%   well-posedness of system \eqref{eq1.1}--\eqref{eq1.4} with initial
%   data $u_{0},\nabla d_{0}\in L^{n}(\mathbb{R}^{n})$. Hinmeman and
%   Wang \cite{JHW} established the global well-posedness of system
%   \eqref{eq1.1}--\eqref{eq1.4} in dimension three with small initial
%    data $(u_{0},\nabla d_{0})$ in $L^{3}_{uloc}$, where $L^{3}_{uloc}$
%    is the space of uniformly locally $L^{3}$-integrable functions in
%    $\mathbb{R}^{3}$.
Recently, Wang in \cite{W} used the framework of
Koch and Tataru \cite{HD}  to proved that if the initial data
$(u_{0},d_{0})\in BMO^{-1}\times BMO$ with small norm, then system
\eqref{eq1.1}--\eqref{eq1.4} exists a global-in time Koch-Tataru
type solution. For the regularity issues of solutions to system
\eqref{eq1.1}--\eqref{eq1.4}, in \cite{DW1,DW2},  Du and Wang used
the frameworks of the Germain, Pavlovi\'{c} and Staffilani
\cite{GPS} and the Dong and Du \cite{DD} to study the rgularity of
the Koch-Tataru type solutions to \eqref{eq1.1}--\eqref{eq1.4}, and
proved that the solution obtained in \cite{W} has arbitrary spatial
and temporal regularity

Since the nematic liquid crystal flows \eqref{eq1.1}--\eqref{eq1.4}
is a strong coupling system between the incompressible NS equations
and the transport heat flow of harmonic maps, there are some similar
properties between the NS eqations and the nematic liquid crystal
flows.  In fact, when researchers studied system
\eqref{eq1.1}--\eqref{eq1.4}, the solutions to system
\eqref{eq1.1}--\eqref{eq1.4} share the similar scaling properties of
solutions to the Navier--Stokes equations. That is, if
$(u(x,t),d(x,t),P(x,t))$ solves \eqref{eq1.1}--\eqref{eq1.4}, then
\begin{align*}
(u_{\lambda}(x,t),d_{\lambda}(x,t),P_{\lambda}(x,t)):=(\lambda
u(\lambda x,\lambda^{2} t),d(\lambda x,\lambda^{2}t),
\lambda^{2}P(\lambda x, \lambda^{2}t))
\end{align*}
for any $\lambda>0$ is also a solution of
\eqref{eq1.1}--\eqref{eq1.3} with the initial data
$(u_{0\lambda}(x),d_{0\lambda}(x)):=(\lambda u_{0}(\lambda
x),d_{0}(\lambda x))$. These  useful properties for the system
\eqref{eq1.1}--\eqref{eq1.4} lead to the following definition. A
function space $(X,Y)$ called a critical space for
\eqref{eq1.1}--\eqref{eq1.4} if it is invariant under the scaling
\begin{align*}
(f_{\lambda}(x),g_{\lambda}(x)):=(\lambda f(\lambda x), g(\lambda
x)), \text{ for all } (f,g)\in X\times Y.
\end{align*}
It is easy to varify that the spaces $L^{n}(\mathbb{R}^{n})\times
\dot{W}^{1,n}(\mathbb{R}^{n})$ and $BMO^{-1}\times BMO$ are critical
spaces for system \eqref{eq1.1}--\eqref{eq1.4}. We also notice that
$BMO^{-1}\times BMO$ may be regard as  the largest critical space
for initial data, where such well-posedness  of  system
\eqref{eq1.1}--\eqref{eq1.4} can be constructed.

Motivated by the works of Koch and Tatatu \cite{HD}, Germain,
Pavlovi\'{c} and Staffilani \cite{GPS}, Dong and Du \cite{DD} and
Zhang et al. \cite{ZZ} on the NS equations, and by the works of  Du
and Wang \cite{DW1,DW2}, and Wang \cite{W} on the nematic liquid
crystal flows,
 in the present paper, we shall consider the
decay estimates for any space-time derivative of the Koch-Tataru
type solution involving some borderline Besov space norms. As a
corollary, we also present the corresponding decay estimate in time.
Though this work is partially enlighted by the paper \cite{ZZ} of
Zhang et al., who considered the regularity of Koch-Tataru solution
to the NS equations, we have a more sophisticated system to estimate
due to the coupling between the velocity field $u$ and the
orientation $d$. Moreover, we overcome the difficulties caused by
the operator $\partial_{t}$ in our proofs. On the other hand,
inspired by papers \cite{C1,C2} of Chemin, we shall study the
trajectories of the Koch-Tataru type solution to
\eqref{eq1.1}--\eqref{eq1.4}. In order to state our main results, we
first recall the definition of the spaces $BMO$, $BMO^{-1}$ and the
existence space of Koch-Tataru type solutions to the nematic liquid
crystal flows, for more details about these space, we refer to
\cite{DW1,DW2,HD,PG,W,ZZ} and the references therein.

\begin{definition} \label{def1.1}
\ 1. Let $W$ be the solution of $W_{t}-\Delta W=0$ with initial data
$f$. Denote
\begin{align*}
&[f]_{BMO}:=\sup_{0<r\leq +\infty; x\in \mathbb{R}^{n}} \left(r^{-n}
\int_{0}^{r^{2}}\int_{|y-x|<r} |\nabla W|^{2}
\text{d}y\text{d}t\right)^{\frac{1}{2}};\nonumber\\
%%%%--------------------------------------------------------------------
&\|f\|_{BMO^{-1}}:=\sup_{0<r\leq +\infty; x\in \mathbb{R}^{n}}
\left(r^{-n} \int_{0}^{r^{2}}\int_{|y-x|<r} | W|^{2}
\text{d}y\text{d}t\right)^{\frac{1}{2}}.
\end{align*}
We say a function $f\in L^{1}_{loc}(\mathbb{R}^{n})$ is in
$BMO(\mathbb{R}^{n})$ if the semi-norm $[f]_{BMO(\mathbb{R}^{n})}$
is finite, and we say $f$ is in $BMO^{-1}(\mathbb{R}^{n})$ if the
norm $\|f\|_{BMO^{-1}}$ is finite. Clearly the divergence of a
vector field with components in $BMO(\mathbb{R}^{n})$ is in
$BMO^{-1}(\mathbb{R}^{n})$, i.e., $\nabla \cdot
(BMO(\mathbb{R}^{n}))^{n}=BMO^{-1}(\mathbb{R}^{n})$.
%See more details in paper of Koch and Tataru \cite{HD}, and the monograph of
%Lemari\'{e}-Rieusset \cite{PG}.
\medskip

\ 2.  We say  a function $f$ defined on $\mathbb{R}^{n+1}_{+}$
belongs to the space $X(\mathbb{R}^{n})$ provided
\begin{align*}
\||f|\|_{X}: &\triangleq \sup_{0<t\leq \infty}
\|f(\cdot,t)\|_{L^{\infty}} +\|f\|_{X}<+\infty,
\end{align*}
 with
\begin{align*}
\|f\|_{X}:&\triangleq\sup_{0<t\leq +\infty}\sqrt{t}\|\nabla
f(\cdot,t)\|_{L^{\infty}}+\sup_{x\in\mathbb{R}^{n},0<r^{2}\leq
+\infty}\left(r^{-n}\int_{0}^{r^{2}}\int_{|y-x|<r} |\nabla
f(y,t)|^{2}\text{d}y\text{d}t\right)^{\frac{1}{2}}.
\end{align*}
We say  a function $g$ defined on $\mathbb{R}^{n+1}_{+}$ belongs to
the space $Z(\mathbb{R}^{n})$ provided
\begin{align*}
\|g\|_{Z}:&\triangleq\sup_{0<t\leq
+\infty}\sqrt{t}\|g(\cdot,t)\|_{L^{\infty}}+\sup_{x\in\mathbb{R}^{n},0<r^{2}\leq
+\infty}\left(r^{-n}\int_{0}^{r^{2}}\int_{|y-x|<r}
|g(y,t)|^{2}\text{d}y\text{d}t\right)^{\frac{1}{2}}<\infty.
\end{align*}
\end{definition}
%--------------------(def1.1)------------------------------------

It is easy to see that  $(X,\||\cdot|\|_{X})$ and
$(Z,\|\cdot\|_{Z})$ are two Banach space, and we say $Z\times X$ is
the Koch-Tataru existence space for the nematic liquid crystal flow
\eqref{eq1.1}--\eqref{eq1.4}. We also need to recall the following
existence and regularity results for the nematic liquid crystal
flows with initial data $(u_{0},d_{0})$ in $BMO^{-1}\times BMO$,
which was proved by Wang \cite{W}, and Du and Wang \cite{DW1,DW2}.

\begin{theorem}\label{thm1.2}
(see  Theorems 1.6 of \cite{W} and Theorem 1.5 of \cite{DW1,DW2})
There exists $\varepsilon=\varepsilon(n)>0$ small enough, and
$C_{0}>0$, such that if $u_{0}\in BMO^{-1}$ with $\nabla\cdot
u_{0}=0 $, and $d_{0}\in BMO$ satisfies
\begin{align*}
\|u_{0}\|_{BMO^{-1}}+[d_{0}]_{BMO}\leq \varepsilon,
\end{align*}
then there exists a unique global-in-time solution $(u,d)\in Z\times
X$ to the nematic liquid crystal flow  \eqref{eq1.1}--\eqref{eq1.4}
so that
\begin{align*}
\|u\|_{Z}+\|d\|_{X} \leq C_{0}\varepsilon.
\end{align*}
 Moreover, for any integers $k,m\geq 0$, there exists a positive constant
$C_{k,m}$,  such that  the unique global-in-time solution $(u,d)$
satisfies
\begin{align}\label{eq1.5}
\|t^{k+\frac{m}{2}}\partial_{t}^{k}\nabla^{m}
u\|_{Z}+\|t^{k+\frac{m}{2}}\partial_{t}^{k}\nabla^{m} d\|_{X}\leq
C_{k,m} \varepsilon.
\end{align}
%%----------------------------(eq1.5)------------------------------------
\end{theorem}
%%----------------------------(thm1.2)------------------------------------

Now we present our main results as follows.

\begin{theorem}\label{thm1.3}
There exists $\varepsilon=\varepsilon(n)>0$ small enough, if
$u_{0}\in BMO^{-1}$ with $\nabla\cdot u_{0}=0 $, and $d_{0}\in BMO$
satisfies
\begin{align*}
\|u_{0}\|_{BMO^{-1}}+[d_{0}]_{BMO}\leq \varepsilon,
\end{align*}
then for any integer $k,m\geq 0$, there exists a constant $C_{k,m}$,
such that the global-in-time solution $(u,d)$ to the nematic liquid
crystal flow \eqref{eq1.1}--\eqref{eq1.4} present by Theorem
\ref{thm1.2} satisfies
\begin{align}\label{eq1.6}
\|t^{k+\frac{m}{2}}(\partial_{t}^{k}\nabla^{m}
u,\partial_{t}^{k}\nabla^{m}\nabla
d)\|_{\widetilde{L}^{\infty}(\mathbb{R}_{+};\dot{B}^{-1}_{\infty,\infty})\cap\widetilde{L}^{1}(\mathbb{R}_{+};\dot{B}^{1}_{\infty,\infty})}
\leq C_{m,k} \varepsilon.
\end{align}
%%----------------------------(eq1.6)------------------------------------

\end{theorem}
%%----------------------------(thm1.3)------------------------------------

Obviously, Theorem \ref{thm1.3} implies the following decay estimate
immediately:

\begin{corollary}\label{cor1.4}
Under the assumptions of Theorem \ref{thm1.3}, the global-in-time
 solution $(u,d)$  to the nematic liquid
crystal flow \eqref{eq1.1}--\eqref{eq1.4} present by Theorem
\ref{thm1.2} satisfies
\begin{align*}
\|(\partial_{t}^{k}\nabla^{m} u,\partial_{t}^{k}\nabla^{m}\nabla
d)\|_{\dot{B}^{-1}_{\infty,\infty}}\leq C_{k,m} t^{-\frac{m}{2}-k},
\end{align*}
 for all $t\in\mathbb{R}_{+}$ and $k,m\geq 0$.

\end{corollary}
%%----------------------------(cor1.4)------------------------------------

By using Theorem 3.2 from Chemin \cite{C1} (or Theorem 3.2.2 from
\cite{C2}) and the result of the above Theorem \ref{thm1.3} for the
case $k=m=0$, we have the following result concerning the
trajectories of the Koch-Tataru  solution to
\eqref{eq1.1}--\eqref{eq1.4} present by Theorem \ref{thm1.2}.

\begin{theorem}\label{thm1.5}
Under the assumptions of Theorem \ref{thm1.3}, there exists a unique
continuous mapping $\gamma$ from $\mathbb{R}^{n+1}_{+}$ to
$\mathbb{R}^{n}$ such that
\begin{align*}
\gamma(x,t)=x+\int_{0}^{t} (u(\gamma(x,\tau),\tau),\nabla
d(\gamma(x,\tau),\tau))\text{d}\tau
\end{align*}
and
\begin{align}\label{eq1.7}
|\gamma(x_{1},t)-\gamma(x_{2},t)|\leq C
|x_{1}-x_{2}|^{1-C\varepsilon}\exp\left(C\varepsilon t\right)
\end{align}
%%----------------------------(eq1.7)------------------------------------
for all $t>0$ and $|x_{1}-x_{2}|\leq 1$.

\end{theorem}
%%----------------------------(thm1.5)------------------------------------

The remaining parts of the present paper are organized as follows.
 In section 2, we
introduce the Littlewood-Paley decomposition, the definition of
Besov space, and some useful estimates of the semi-group of the heat
equation on distributions the Fourier transform of which is
supported in a rang. In section 3, by using Fourier local analysis
method, we give the proof of Theorem \ref{thm1.3}. The last section
is devoted to proving Theorem \ref{thm1.5}. Throughout this paper,
we use $\|\cdot\|_{X}$ to denote the norm of the scalar
$X$-functions or the norm of the $n$-vector $X$-functions. We also
denote by $C$ the various positive constant only depending on the
dimension number $n$, in particular, we denote by $C=C_{a,b,\cdots}$
the positive constant which depends only on the dimension number $n$
and the indicated parameters $a$, $b$, $\cdots$. All such constants
may vary from line to line.

\section{Preliminaries}\label{Pre}

In this section, we are going to recall the dyadic partition of
unity in the Fourier variable, the so-called, Littlewood--Paley
theory, the definition of Besov space, and some classical results on
the  heat operator. Part of the materials presented here can be
found in \cite{C1,C2,PG,XZ,ZZ}. Let $\mathcal{S}(\mathbb{R}^{n})$ be
the Schwartz class of rapidly decreasing functions.
 Let  $\varphi \in
\mathcal{S}(\mathbb{R}^{3})$ with values in $[0,1]$ such that
$\varphi$ is supported in
$\mathfrak{C}=\{\xi\in\mathbb{R}^{n}:\frac{3}{4}\leq |\xi|\leq
\frac{8}{3}\}$ and
\begin{align*}
\sum_{j\in\mathbb{Z}}\varphi_{j}(\xi)=1 \quad\text{ for any
}\xi\in\mathbb{R}^{n}\backslash\{0\},
\end{align*}
where $\varphi_{j}(\xi):\triangleq\varphi(2^{-j}\xi)$. Denoting
$h_{j}=\mathcal{F}^{-1}\varphi_{j}$, we define the homogeneous
dyadic blocks as
\begin{align}\label{eq2.1}
&{\Delta}_{j} u:\triangleq\varphi(2^{-j}D)u=\int_{\mathbb{R}^{n}}
h_{j}(y)u(x-y)\text{d}y \quad\text{ and }
{S}_{j}f:\triangleq\sum_{k\leq j-1}\Delta_{j}f,\quad \forall j\in
\mathbb{Z},
\end{align}
%--------------------------(eq2.1)----------------------
where $D=(D_{1},D_{2},\cdots,D_{n})$ and
$D_{j}=i^{-1}\partial_{x_{j}}, (j=1,2,3)$ with $i^{2}=-1$. Let
$S'_{h}(\mathbb{R}^{n})$ be the space of tempered distributions $u$
such that
\begin{align*}
\lim_{\lambda\rightarrow\infty}\|\theta(\lambda D)
u\|_{L^{\infty}}=0 \text{ for any } \theta\in
\mathcal{D}(\mathbb{R}^{n}),
\end{align*}
where $\mathcal{D}(\mathbb{R}^{n})$ is the space of smooth compactly
supported functions on $\mathbb{R}^{n}$.
 Then we have the formal decomposition
\begin{align*}
u=\sum_{j\in\mathbb{Z}}\Delta_{j} u, \forall
u\in\mathcal{S}'_{h}(\mathbb{R}^{n}).
\end{align*}
Informally, ${\Delta}_{j}={S}_{j+1}-{S}_{j}$ is a frequency
projection to the annulus $\{ |\xi|\approx 2^{j}\}$, while $S_{j}$
is the frequency projection to $\{0<|\xi|\lesssim 2^{j}\}$.  One
easily verifies that with the above choice of $\varphi$,
\begin{align*}
\Delta_{j}\Delta_{k}f\equiv 0\text{ if }|j-k|\geq 2 \text{ and }
\Delta_{j}(S_{k-1}f\Delta_{k}f )\equiv 0\text{ if } |j-k|\geq 5.
\end{align*}
We recall now the definition of the homogeneous Besov spaces from
\cite{C2,PG}.

\begin{definition}\label{def2.1}
(homogeneous Besov spaces $\dot{B}^{s}_{p,r}(\mathbb{R}^{n})$) Let
$s\in\mathbb{R}$, $1\leq p,r\leq \infty$, we set
\begin{equation*}
\|u\|_{{\dot{B}}^{s}_{p,r}} :\triangleq \left\{
\begin{array}{l}
\left(\sum_{j\in\mathbb{Z}}
2^{jsr}\|\Delta_{j}u\|_{L^{p}}^{r}\right)^{\frac{1}{r}}
\quad\text{ for } 1\leq r< \infty, \\
\sup_{j\in\mathbb{Z}}
2^{js}\|\Delta_{j}u\|_{L^{p}}\quad\quad\quad\text{ for }r=\infty.
\end{array}
\right.
\end{equation*}
\begin{itemize}
\item For $s<\frac{n}{p}$ (or $s=\frac{n}{p}$ if $r=1$), we define
$\dot{B}^{s}_{p,r}(\mathbb{R}^{n})\triangleq
\{u\in\mathcal{S}'_{h}(\mathbb{R}^{n})|
\|u\|_{\dot{B}^{s}_{p,r}}<\infty\}.$

\item If $k\in\mathbb{N}$ and $\frac{n}{p}+k\leq s<\frac{n}{p}+k+1$
(or $s=\frac{n}{p}+k+1$ if $r=1$), the
$\dot{B}^{s}_{p,r}(\mathbb{R}^{n})$ is defined as the subset of
distributions $u\in \mathcal{S}'_{h}(\mathbb{R}^{n})$ such that
$\partial_{x}^{\gamma} u\in \dot{B}^{s-k}_{p,r}(\mathbb{R}^{n})$
whenever $\partial^{\gamma}_{x}:=
\partial^{\gamma_{1}}_{x_{1}}\partial^{\gamma_{2}}_{x_{2}}\cdots
\partial^{\gamma_{n}}_{x_{n}}$ with $\gamma= (\gamma_{1},\gamma_{2},\cdots,\gamma_{n})$ and $|\gamma|=k$. For short, we just denote the
space $\dot{B}^{s}_{p,r}(\mathbb{R}^{n})$ by $\dot{B}^{s}_{p,r}$.
\end{itemize}
\end{definition}
%-------------------------(def2.1)--------------------------

We still need to define the so called Chemin-Lerner type spaces
$\widetilde{L}^{\rho}_{T}(\dot{B}^{s}_{p,r}(\mathbb{R}^{n}))$, i.e.,

\begin{definition}\label{def2.2}
Let  $s\in\mathbb{R}$, $1\leq p,r,\rho\leq \infty$ and $T\in
(0,+\infty]$, we define
$\widetilde{L}^{\rho}_{T}(\dot{B}^{s}_{p,r}(\mathbb{R}^{n}))$ as the
completion of $C([0,T];\mathcal{S}(\mathbb{R}^{3}))$ by the norm
\begin{align*}
\|f\|_{\widetilde{L}^{\rho}_{T}(\dot{B}^{s}_{p,r})} :\triangleq
\left(\sum_{j\in\mathbb{Z}}2^{jsr}\|\Delta_{j}
f\|_{L^{\rho}_{T}(L^{p})}^{r}\right)^{\frac{1}{r}},
\end{align*}
with the usual change if $r=\infty$. For short, we just denote this
space by $\widetilde{L}^{\rho}_{T}(\dot{B}^{s}_{p,r})$.
\end{definition}
%-------------------------(def2.2)--------------------------

\begin{remark}\label{rem2.3}
By virtue of the Minkowski inequality, one obtains
\begin{align*}
&\|u\|_{\widetilde{L}^{\rho}_{T}(\dot{B}^{s}_{p,r})}\leq
\|u\|_{L^{\rho}_{T}(\dot{B}^{s}_{p,r})} \text{ if } r\geq \rho,
\text{ and }
\quad\|u\|_{\widetilde{L}^{\rho}_{T}(\dot{B}^{s}_{p,r})}\geq
\|u\|_{L^{\rho}_{T}(\dot{B}^{s}_{p,r})} \text{ if } r\leq \rho.
\end{align*}
\end{remark}
%-------------------------(rem2.3)--------------------------
\medskip

The following lemma gives the way the product acts on Chemin-Lerner
type spaces.

\begin{lemma}\label{lem2.4}
Let $s_{1},s_{2}\in \mathbb{R}$,  $p, p_{1}, p_{2}, q, q_{1}, q_{2},
r_{1},r_{2}\in [1,\infty]$ such that
\begin{align*}
\frac{1}{q_{1}}+\frac{1}{q_{2}}=\frac{1}{q},
\frac{1}{r_{1}}+\frac{1}{r_{2}} =\frac{1}{r}\leq 1 \text{  and }
p\geq \max\{p_{1}, p_{2}\}
\end{align*}
If $s_{1}+s_{2}>0$, the product maps
$\widetilde{L}^{q_{1}}_{T}(\dot{B}^{s_{1}}_{p_{1},r_{1}})\times
\widetilde{L}^{q_{2}}_{T}(\dot{B}^{s_{2}}_{p_{2},r_{2}})$ into
$\widetilde{L}^{q}_{T}(\dot{B}^{s_{1}+s_{2}-n(\frac{1}{p_{1}}+\frac{1}{p_{2}}-\frac{1}{p})}_{p,r})$.
If $s_{1}=-s_{2}$ and $r=1$, the product maps
$\widetilde{L}^{q_{1}}_{T}(\dot{B}^{s_{1}}_{p_{1},r_{1}})\times
\widetilde{L}^{q_{2}}_{T}(\dot{B}^{s_{2}}_{p_{2},r_{2}})$ into
$\widetilde{L}^{q}_{T}(\dot{B}^{-n(\frac{1}{p_{1}}+\frac{1}{p_{2}}-\frac{1}{p})}_{p,\infty})$.

%i.e., there holds
%\begin{align}\label{eq2.3}
%\|ab\|_{\widetilde{L}^{q}_{T}(\dot{B}^{s_{1}+s_{2}-n(\frac{1}{p_{1}}+\frac{1}{p_{2}}-\frac{1}{p})}_{p,r})}\lesssim
%\|a\|_{\widetilde{L}^{q_{1}}_{T}(\dot{B}^{s_{1}}_{p_{1},r_{1}})}\|b\|_{\widetilde{L}^{q_{2}}_{T}(\dot{B}^{s_{2}}_{p_{2},r_{2}})}
%\end{align}
%%---------------------------(eq2.3)---------------------------------------
%for all $a\in
%\widetilde{L}^{q_{1}}_{T}(\dot{B}^{s_{1}}_{p_{1},r_{1}})$ and  $b\in
%\widetilde{L}^{q_{2}}_{T}(\dot{B}^{s_{2}}_{p_{2},r_{2}})$.
\end{lemma}
%-----------------(lem2.4)---------------------

We end the section by the following lemma, which gives useful
estimates for the semi-group of the neat equation restricted to
functions with compact supports away from the origin in Fourier
transform variables (see Lemma 2.1 in \cite{C1}, Lemma 2.1.2 in
\cite{C2}, and Lemmas 2.1 and 3.1 in \cite{ZZ}).

 %.   it gives that the results of the action of the semi-group of the heat equation on distributions the Fourier
  %  transform of which is supported in a rang

\begin{lemma}\label{lem2.5}
Let $\varphi$ be the smooth function defined as the beginning of
this section, and $\widetilde{\varphi}$ be the smooth function
supported in the ring $\widetilde{\mathfrak{C}}=
\{\xi\in\mathbb{R}^{n}; \frac{3}{8}\leq |\xi|\leq \frac{10}{3}\}$
such that $\widetilde{\varphi}\varphi=\varphi$. Let $1\leq i.j,k\leq
n$, $m\geq 0$
\begin{align*}
&g^{i,j,k}_{q}(x,t):\triangleq \int_{\mathbb{R}^{n}} e^{i x\cdot\xi}
\varphi(2^{-q}\xi) e^{-t|\xi|^{2}}
\left(\delta_{ij}-\frac{\xi_{i}\xi_{j}}{|\xi|^{2}}\right)\xi_{k}\text{d}\xi,\nonumber\\
%%%%-----------------------------------------------------------------------------
&g_{1,q}^{i,j}(x,t):\triangleq \int_{\mathbb{R}^{n}} e^{i x\cdot\xi}
\varphi(2^{-q}\xi) e^{-t|\xi|^{2}}
\left(\delta_{ij}-\frac{\xi_{i}\xi_{j}}{|\xi|^{2}}\right)\text{d}\xi,\nonumber\\
%%%%-------------------------------------------------------------------------------
&g_{2,q}^{i,j}(x,t):\triangleq \int_{\mathbb{R}^{n}} e^{i x\cdot\xi}
\varphi(2^{-q}\xi) e^{-t|\xi|^{2}}\text{d}\xi
\end{align*}
and
\begin{align*}
g_{3,q}(x,t) :\triangleq \int_{\mathbb{R}^{n}} e^{i x\cdot\xi}
\widetilde{\varphi}(2^{-q}\xi)t^{\frac{m}{2}} \xi^{\gamma}
e^{-t|\xi|^{2}}\text{d}\xi,
\end{align*}
where $|\gamma|=m$. Then there exist three positive constants $c$,
$C$ and $C_{m}$such that, for all $q\in\mathbb{Z}$ and $t>0$, we
have
\begin{align*}
&|g_{q}^{i,j,k}(x,t)|\leq C \frac{2^{q(n+1)}}{1+|2^{q}x|^{2n}}
e^{-ct2^{2 q}};\\
%%%%----------------------------(eq3.1)------------------------------------
&|t^{\frac{m}{2}}\nabla^{m} g_{q}^{i,j,k}(x,t)|\leq C_{m}
\frac{2^{q(n+1)}}{1+|2^{q}x|^{2n}} e^{-ct2^{2
q}};\\
%%----------------------------(eq3.2)------------------------------------
&|g_{1,q}^{i,j}(x,t)|+|g_{2,q}(x,t)|\leq C
\frac{2^{qn}}{1+|2^{q}x|^{2n}} e^{-ct2^{2
q}};\\
%%----------------------------(eq3.3)------------------------------------
&|t^{\frac{m}{2}}\nabla^{m} g_{2,q}(x,t)|+|g_{3,q}(x,t)|\leq C_{m}
\frac{2^{qn}}{1+|2^{q}x|^{2n}} e^{-ct2^{2 q}}.
%%----------------------------(eq3.4)------------------------------------
\end{align*}
\end{lemma}
%%----------------------------(lem2.5)------------------------------------

\section{The proof of Theorem \ref{thm1.3}}

In this section, we shall give the proof of Theorem \ref{thm1.3}. We
first give the following lemma, which gives some estimates for the
linear heat equations.

\begin{lemma}\label{lem3.1}
Taking a constant $\varepsilon=\varepsilon(n)>0$ small enough, and
the initial data $(u_{0},d_{0})\in BMO^{-1}\times BMO$ with
$\|u_{0}\|_{BMO^{-1}}+[d_{0}]_{BMO}\leq \varepsilon$, the for any
integer $m\geq 0$, we have
\begin{align}\label{eq3.1}
\|t^{\frac{m}{2}}(\nabla^{m}e^{t\Delta} u_{0},
\nabla^{m+1}e^{t\Delta} d_{0}
)\|_{\widetilde{L}^{\infty}(\mathbb{R}_{+};\dot{B}^{-1}_{\infty,\infty})\cap
\widetilde{L}^{1}(\mathbb{R}_{+};\dot{B}^{1}_{\infty,\infty})}\leq
C_{m}\varepsilon.
\end{align}
%%----------------------------(eq3.1)------------------------------------
\end{lemma}
%%----------------------------(lem3.1)------------------------------------

\begin{proof}\label{proof of lem3.1}
At the beginning, we need to recall Lamma 2.2 in  \cite{DW1},  Du
and Wang proved that under the assumptions of Lemma \ref{lem3.1},
there holds
\begin{align}\label{eq3.2}
\|t^{\frac{m}{2}}\nabla^{m}e^{t\Delta} u_{0}\|_{Z}\leq C_{m}
\|u_{0}\|_{BMO^{-1}}\quad \text{ and }\quad \|
t^{\frac{m}{2}}\nabla^{m}e^{t\Delta} d_{0} \|_{ X}\leq
C_{m}[d_{0}]_{BMO},
\end{align}
%%----------------------------(eq3.2)------------------------------------
where the norms $\|\cdot\|_{Z}$ and $\|\cdot\|_{X}$ defined as
Definition \ref{def1.1}.

In what follows, we are in a position to the proof of \eqref{eq3.1}.
Firstly, we prove that
\begin{align}\label{eq3.3}
\|t^{\frac{m}{2}}\nabla^{m}e^{t\Delta}
u_{0}\|_{\widetilde{L}^{\infty}(\mathbb{R}_{+};\dot{B}^{-1}_{\infty,\infty})\cap
\widetilde{L}^{1}(\mathbb{R}_{+};\dot{B}^{1}_{\infty,\infty})}\leq
C_{m}\|u_{0}\|_{BMO^{-1}}.
\end{align}
%%----------------------------(eq3.3)------------------------------------
Notice that for all $t\in [2^{-2q},\infty)$, it follows that
\begin{align}\label{eq3.4}
\!2^{-q} \|\Delta_{q} t^{\frac{m}{2}}\nabla^{m} e^{t\Delta}
u_{0}\|_{L^{\!\infty}}\!\leq 2^{-q} \| t^{\frac{m}{2}}\nabla^{m}
e^{t\Delta} u_{0}\|_{L^{\!\infty}}\!\leq 2^{-q} t^{-\frac{1}{2}} \|
t^{\frac{m}{2}}\nabla^{m} e^{t\Delta} u_{0}\|_{Z} \leq C_{m}
\|u_{0}\|_{BMO^{-1}},
\end{align}
%%----------------------------(eq3.4)------------------------------------
where we have used \eqref{eq3.2} in the last inequality. For $t\in
(0, 2^{-2q})$, by using the embedding $BMO^{-1}\hookrightarrow
\dot{B}^{-1}_{\infty,\infty}$ and Lemma \ref{lem2.5}, we have
\begin{align}\label{eq3.5}
\!2^{-q} \|\Delta_{q} t^{\frac{m}{2}}\nabla^{m} e^{t\Delta}
u_{0}\|_{L^{\!\infty}}=&2^{-q} \|g_{3,q}\ast \Delta_{q}
u_{0}\|_{L^{\infty}}\leq C_{m} 2^{-q} \|\Delta_{q}
u_{0}\|_{L^{\infty}}\nonumber\\
\leq& C_{m} \|u_{0}\|_{\dot{B}^{-1}_{\infty,\infty}}\leq
C_{m}\|u_{0}\|_{BMO^{-1}}.
\end{align}
%%----------------------------(eq3.5)------------------------------------
On the other hand, again by $BMO^{-1}\hookrightarrow
\dot{B}^{-1}_{\infty,\infty}$ and Lemma \ref{lem2.5}, we get
\begin{align*}
2^{q} \|\Delta_{q} t^{\frac{m}{2}} \nabla^{m} e^{t\Delta}
u_{0}\|_{L^{1}(\mathbb{R}_{+};L^{\infty})}=&2^{q} \int_{0}^{\infty}
\|g_{3,q}\ast \Delta_{q} u_{0}\|_{L^{\infty}}\text{d}t \leq C_{m}
2^{q}\int_{0}^{\infty} e^{-ct 2^{2q}}\text{d}t \|\Delta_{q}
u_{0}\|_{L^{\infty}}\nonumber\\
 \leq& C_{m} 2^{-q} \|\Delta_{q}
u_{0}\|_{L^{\infty}} \leq
C_{m}\|u_{0}\|_{\dot{B}^{-1}_{\infty,\infty}}\leq
C_{m}\|u_{0}\|_{BMO^{-1}},
\end{align*}
which together with \eqref{eq3.4} and \eqref{eq3.5} implies that
\eqref{eq3.3}. In a similar way, we can prove the rest part of
\eqref{eq3.2}. In fact, for all $t\in [2^{-2q},\infty)$, we have
\begin{align}\label{eq3.6}
\!2^{-q} \|\Delta_{q} t^{\frac{m}{2}}\nabla^{m+1}\! e^{t\Delta}
d_{0}\|_{L^{\!\infty}}\!\leq\! 2^{-q} \|\nabla(
t^{\frac{m}{2}}\nabla^{m}\! e^{t\Delta}
d_{0})\|_{L^{\!\infty}}\!\leq \!2^{-q} t^{-\frac{1}{2}} \|
t^{\frac{m}{2}}\nabla^{m}\! e^{t\Delta} d_{0}\|_{X} \!\leq\! C_{m}
[d_{0}]_{BMO},
\end{align}
%%----------------------------(eq3.6)------------------------------------
and for $t\in (0, 2^{-2q})$, by using the embedding
$BMO\hookrightarrow \dot{B}^{0}_{\infty,\infty}$ and Lemma
\ref{lem2.5}, we have
\begin{align}\label{eq3.7}
\!2^{-q} \|\Delta_{q} t^{\frac{m}{2}}\nabla^{m+1} e^{t\Delta}
d_{0}\|_{L^{\!\infty}}=&2^{-q} \|g_{3,q}\ast \Delta_{q}\nabla
d_{0}\|_{L^{\infty}}\leq C_{m} 2^{-q} \|\Delta_{q}
\nabla d_{0}\|_{L^{\infty}}\nonumber\\
\leq& C_{m} \|\Delta_{q} d_{0}\|_{L^{\infty}}\leq C_{m}
\|d_{0}\|_{\dot{B}^{0}_{\infty,\infty}}\leq C_{m}[d_{0}]_{BMO}.
\end{align}
%%----------------------------(eq3.7)------------------------------------
Again by $BMO\hookrightarrow \dot{B}^{0}_{\infty,\infty}$ and Lemma
\ref{lem2.5}, we get
\begin{align}\label{eq3.8}
2^{q} \|\Delta_{q} t^{\frac{m}{2}} \nabla^{m+1} e^{t\Delta}
d_{0}\|_{L^{1}(\mathbb{R}_{+};L^{\infty})}=&2^{q} \int_{0}^{\infty}
\|g_{3,q}\ast \Delta_{q}\nabla d_{0}\|_{L^{\infty}}\text{d}t \leq
C_{m} 2^{q}\int_{0}^{\infty} e^{-ct 2^{2q}}\text{d}t
\|\Delta_{q}\nabla
d_{0}\|_{L^{\infty}}\nonumber\\
 \leq& C_{m} 2^{-q} \|\Delta_{q}\nabla
d_{0}\|_{L^{\infty}} \leq
C_{m}\|d_{0}\|_{\dot{B}^{0}_{\infty,\infty}}\leq C_{m}[d_{0}]_{BMO},
\end{align}
%%----------------------------(eq3.8)------------------------------------
Combining \eqref{eq3.6}, \eqref{eq3.7} and \eqref{eq3.8} together,
it follows that
\begin{align*}
\|t^{\frac{m}{2}}\nabla^{m+1}e^{t\Delta}
d_{0}\|_{\widetilde{L}^{\infty}(\mathbb{R}_{+};\dot{B}^{-1}_{\infty,\infty})\cap
\widetilde{L}^{1}(\mathbb{R}_{+};\dot{B}^{1}_{\infty,\infty})}\leq
C_{m}[d_{0}]_{BMO},
\end{align*}
which together with \eqref{eq3.3} verifies \eqref{eq3.1}. This
completes the proof Lemma \ref{lem3.1}.
\end{proof}
%%----------------------------(proof of lem3.1)------------------------------------
\medskip

We are now in a position to the proof of Theorem \ref{thm1.3}. We
first rewrite the system \eqref{eq1.1}--\eqref{eq1.4} as an integral
system:
\begin{equation}\label{eq3.9}
\left\{
\begin{array}{l}
u:=e^{-t\Delta} u_{0}+\mathbb{T}_{1}(u,d)(x,t), \\
d:=e^{-t\Delta } d_{0}+\mathbb{T}_{2}(u,d)(x,t),
\end{array}
\right.
\end{equation}
%%----------------------------(eq3.9)------------------------------------
with
\begin{equation*}
\left\{
\begin{array}{l}
\mathbb{T}_{1}(u,d)(x,t):=-\int_{0}^{t} e^{-(t-\tau)\Delta}
\mathbb{P}\nabla\cdot(u\otimes
u+\nabla d\odot\nabla d)(\cdot,\tau)\text{d}\tau, \\
\mathbb{T}_{2}(u,d)(x,t):=\int_{0}^{t}e^{-(t-\tau)\Delta} (|\nabla
d|^{2}d-u\cdot\nabla d)(\cdot,\tau)\text{d}t,
\end{array}
\right.
\end{equation*}
 where $\mathbb{P}:=I+\nabla(-\Delta)^{-1}\text{div}$ is the
Helmholtz-Weyl projection operator which has the matrix symbol with
components
\begin{align*}
(\widehat{\mathbb{P}}(\xi))_{j,k}=\delta_{jk}-\xi_{j}\xi_{k}|\xi|^{-2}\quad
\text{with } j,k =1,2,\cdots,n,
\end{align*}
where $\delta_{jk}$ is Kronecker symbol, and $\otimes$ denotes
tensor product.

In what follows, we shall divided the proof of the Theorem
\ref{thm1.3} into two steps.\medskip
\\
\textbf{Step 1.} Estimate \eqref{eq1.6} with the case of $k=0$ and
$m\geq 0$, i.e., we shall prove
\begin{align*}
\|t^{\frac{m}{2}}(\nabla^{m} u, \nabla^{m+1}
d)\|_{\widetilde{L}^{\infty}(\mathbb{R}_{+};\dot{B}^{-1}_{\infty,\infty})\cap
\widetilde{L}^{1}(\mathbb{R}_{+};\dot{B}^{1}_{\infty,\infty})}\leq
C_{m}\varepsilon.
\end{align*}
By using Lemma \ref{lem3.1}, to prove the above inequality,  it
sufficient to verify for any positive integers $m\geq 0$,
\begin{align}\label{eq3.10}
\|t^{\frac{m}{2}} (\nabla^{m} \mathbb{T}_{1}(u,d), \nabla^{m+1}
\mathbb{T}_{2}(u,d))(x,t)\|_{\widetilde{L}^{\infty}(\mathbb{R}_{+};\dot{B}^{-1}_{\infty,\infty})\cap
\widetilde{L}^{1}(\mathbb{R}_{+};\dot{B}^{1}_{\infty,\infty})}\leq
C_{m}\varepsilon.
\end{align}
%%----------------------------(eq3.10)------------------------------------
We first give the estimate of $\mathbb{T}_{1}(u,d)$. For $t\in
[2^{-2q},\infty)$, by using \eqref{eq1.5} and Lemma \ref{lem2.5}, we
have
\begin{align}\label{eq3.11}
2^{-q}\|\Delta_{q}t^{\frac{m}{2}} \nabla^{m}
\mathbb{T}_{1}(u,d)(\cdot,t)\|_{L^{\infty}}\leq&
2^{-q}\|t^{\frac{m}{2}} \nabla^{m}
\mathbb{T}_{1}(u,d)(\cdot,t)\|_{L^{\infty}}\nonumber\\
 \leq&
2^{-q}\|t^{\frac{m}{2}} \nabla^{m} u(\cdot,t)\|_{L^{\infty}}+ 2^{-q}
\|t^{\frac{m}{2}} \nabla^{m} e^{t\Delta}
u_{0}\|_{L^{\infty}}\nonumber\\
\leq& 2^{-q} t^{-\frac{1}{2}} \|t^{\frac{m}{2}} \nabla^{m} u\|_{Z}
+C_{m}\|u_{0}\|_{BMO^{-1}}\nonumber\\
\leq& C_{m}(\varepsilon+\|u_{0}\|_{BMO^{-1}})\leq C_{m}\varepsilon,
\end{align}
%%----------------------------(eq3.11)------------------------------------
where we have used \eqref{eq3.9} in the inequalities above. For
$t\in (0,2^{-2q})$, by denoting
\begin{align*}
\mathbf{F}(x,t):= (u\otimes u +\nabla d\odot \nabla d)(x,t),
\end{align*}
we can rewrite $\Delta_{q} \mathbb{T}_{1}(u,d)$ for any $q\in
\mathbb{Z}$ as
\begin{align*}
\Delta_{q} \mathbb{T}_{1}(u,d) (x,t) =&\int_{0}^{\frac{t}{2}}
\int_{|y|\geq 2^{1-q}} g_{q}(y,t-\tau)\cdot
\mathbf{F}(x-y,\tau)\text{d}y\text{d}\tau\nonumber\\
&+\int_{\frac{t}{2}}^{t} \int_{|y|\geq 2^{1-q}} g_{q}(y,t-\tau)\cdot
\mathbf{F}(x-y,\tau)\text{d}y\text{d}\tau\nonumber\\
&+ \int_{0}^{\frac{t}{2}} \int_{|y|\leq 2^{1-q}}g_{q}(y,t-\tau)\cdot
\mathbf{F}(x-y,\tau)\text{d}y\text{d}\tau\nonumber\\
&+ \int_{\frac{t}{2}}^{t} \int_{|y|\leq 2^{1-q}}g_{q}(y,t-\tau)\cdot
\mathbf{F}(x-y,\tau)\text{d}y\text{d}\tau\nonumber\\
:\triangleq&( F_{1,q}+F_{2,q}+F_{3,q}+F_{4,q})(x,t),
\end{align*}
where $(g_{q}\cdot \mathbf{F})^{i}= g_{q}^{i,j,k}\mathbf{F}^{j,k}$
and $g_{q}^{i,j,k}$ are given by Lemma \ref{lem2.5}. We shall
estimate term by term from $F_{1,q}$ to $F_{4,q}$. Indeed thanks to
\eqref{eq1.5} and Lemma \ref{lem2.5}, we have
\begin{align*}
|t^{\frac{m}{2}}\nabla^{m}F_{1,q}(x,t)|
=&\int_{0}^{\frac{t}{2}}\int_{|y|\geq 2^{1-q}}
(\frac{t}{t-\tau})^{\frac{m}{2}}(t-\tau)^{\frac{m}{2}} (\nabla^{m}
g_{q}) (y,t-\tau)\cdot
\mathbf{F}(x-y,\tau)\text{d}y\text{d}\tau\nonumber\\
\leq& C_{m}\!\! \int_{0}^{\frac{t}{2}} \!\!\int_{|y|\geq 2^{1-q}}\!
\frac{2^{q(n+1)}}{1+(2^{q}|y|)^{2n}} e^{-c(t-\tau)2^{2q}}
|\mathbf{F}(x-y,\tau)|\text{d}y\text{d}\tau\nonumber\\
\leq&C_{m}\!\! \sum_{\ell\in \mathbb{Z}^{n}\backslash\{0\}}
\int_{0}^{t} \int_{y\in
2^{-q}(\ell+[0,1]^{n})}\frac{2^{q(n+1)}}{1+(2^{q}|y|)^{2n}}
|\mathbf{F}(x-y,\tau)|\text{d}y\text{d}\tau\nonumber\\
\leq&C_{m}\!\! \sum_{\ell\in
\mathbb{Z}^{n}\backslash\{0\}}\!\!\frac{2^{q(n+1)}}{|\ell|^{2n}}\!\int_{0}^{2^{-2q}}\!
\int_{y\in 2^{-q}(\ell+[0,1]^{n})}\!\left(|u(x-y,\tau)|^{2}+|\nabla
d(x-y,\tau)|^{2}\right)\text{d}y\text{d}\tau\nonumber\\
\leq& C_{m} \sum_{\ell\in
\mathbb{Z}^{n}\backslash\{0\}}\!\!\frac{2^{q(n+1)}}{|\ell|^{2n}}2^{-q
n} \left(\|u\|_{Z}^{2}+\|d\|_{X}^{2}\right)\leq
C_{m}2^{q}\left(\|u\|_{Z}^{2}+\|d\|_{X}^{2}\right)\leq C_{m} 2^{q}
\varepsilon^{2}.
\end{align*}
Exactly following the same line, we have
\begin{align*}
&|t^{\frac{m}{2}}\nabla^{m}F_{2,q}(x,t)|
=\int_{\frac{t}{2}}^{t}\int_{|y|\geq 2^{1-q}}g_{q}(y,t-\tau)\cdot
t^{\frac{m}{2}} \nabla^{m}
\mathbf{F}(x-y,\tau)\text{d}y\text{d}\tau\nonumber\\
\leq&C_{m} \int_{\frac{t}{2}}^{t} \int_{|y|\geq 2^{1-q}}
\frac{2^{q(n+1)}}{1+(2^{q}y)^{2n}} e^{-c(t-\tau)2^{2q}}
\sum_{i=0}^{m} \left(|t^{\frac{i}{2}}\nabla^{i}
u(x-y,\tau)|^{2}+|t^{\frac{i}{2}}\nabla^{i} \nabla
d(x-y,\tau)|^{2}\right)\text{d}y\text{d}\tau\nonumber\\
\leq&C_{m}
\!\!\sum_{\ell\in\mathbb{Z}^{n}\backslash\{0\}}\!\int_{\frac{t}{2}}^{t}
\int_{y\in 2^{-q}(\ell+[0,1]^{n})}
\frac{2^{q(n+1)}}{1+(2^{q}y)^{2n}} \sum_{i=0}^{m}
\left(|\tau^{\frac{i}{2}}\nabla^{i}
u(x-y,\tau)|^{2}+|\tau^{\frac{i}{2}}\nabla^{i} \nabla
d(x-y,\tau)|^{2}\right)\text{d}y\text{d}\tau\nonumber\\
\leq& C_{m}\!\sum_{\ell\in\mathbb{Z}^{n}\backslash\{0\}}\!
\frac{2^{q(n+1)}}{|\ell|^{2n}} 2^{-qn}\sum_{i=1}^{m}
\left(\|\tau^{\frac{i}{2}}\nabla^{i}u\|_{Z}^{2}+\|\tau^{\frac{i}{2}}\nabla^{i}
d\|_{X}^{2}\right)\leq C_{m} 2^{q} \varepsilon^{2};\nonumber\\
     %%%%%%%%%%%%%%%%%%%%%%%%%%%%%%%%%%%%%%%%%%%%%%%%%%%%%%%%%%%%%%%%%%%%%%%%%%%%%%%%%%%%%%%%%%%%%%%%%%%%%%%55
&|t^{\frac{m}{2}}\nabla^{m}F_{3,q}(x,t)|
=\int_{0}^{\frac{t}{2}}\int_{|y|\leq
2^{1-q}}\left(\frac{t}{t-\tau}\right)^{\frac{m}{2}}(t-\tau)^{\frac{m}{2}}
(\nabla^{m} g_{q})(y,t-\tau)\cdot
\mathbf{F}(x-y,\tau)\text{d}y\text{d}\tau\nonumber\\
\leq& C_{m}\int_{0}^{t} \int_{|y|\leq
2^{1-q}}\frac{2^{q(n+1)}}{1+(2^{q} |y|)^{2n}} e^{-c(t-\tau)2^{2q}}
|\mathbf{F}(x-y,\tau)|\text{d}y\text{d}\tau\nonumber\\
\leq& C_{m} 2^{q(n+1)} \int_{0}^{t} \int_{|z-x|\leq 2^{1-q}}
|\mathbf{F}(z,\tau)|\text{d}z\text{d}\tau\nonumber\\
\leq&C_{m} 2^{q(n+1)} \int_{0}^{2^{-2q}} \int_{|z-x|\leq 2^{1-q}}
\left(|u(z,\tau)|^{2}+|\nabla d(z,\tau)|^{2}\right)\text{d}z\text{d}\tau\nonumber\\
\leq&C_{m} 2^{q(n+1)} 2^{-qn}
\left(\|u\|_{Z}^{2}+\|d\|_{X}^{2}\right)\leq C_{m}
2^{q}\varepsilon^{2}.
\end{align*}
and
\begin{align*}
&|t^{\frac{m}{2}}\nabla^{m}F_{4,q}(x,t)|
=\int_{\frac{t}{2}}^{t}\int_{|y|\leq 2^{1-q}} g_{q}(y,t-\tau)\cdot
t^{\frac{m}{2}}\nabla^{m}
\mathbf{F}(x-y,\tau)\text{d}y\text{d}\tau\nonumber\\
\leq& C_{m}\int_{\frac{t}{2}}^{t} \int_{|y|\leq
2^{1-q}}\frac{2^{q(n+1)}}{1+(2^{q} |y|)^{2n}} e^{-c(t-\tau)2^{2q}}
\sum_{i=0}^{m} \left(|t^{\frac{i}{2}}\nabla^{i}
u(x-y,\tau)|^{2}+|t^{\frac{i}{2}}\nabla^{i} \nabla
d(x-y,\tau)|^{2}\right)\text{d}y\text{d}\tau\nonumber\\
\leq& C_{m}\int_{\frac{t}{2}}^{t} \int_{|y|\leq
2^{1-q}}\frac{2^{q(n+1)}}{1+(2^{q} |y|)^{2n}} e^{-c(t-\tau)2^{2q}}
\sum_{i=0}^{m} \left(|\tau^{\frac{i}{2}}\nabla^{i}
u(x-y,\tau)|^{2}+|\tau^{\frac{i}{2}}\nabla^{i} \nabla
d(x-y,\tau)|^{2}\right)\text{d}y\text{d}\tau\nonumber\\
\leq& C_{m} 2^{q(n+1)} \sum_{i=0}^{m} \int_{0}^{t} \int_{|z-x|\leq
2^{1-q}} \left(|\tau^{\frac{i}{2}}\nabla^{i}
u(z,\tau)|^{2}+|\tau^{\frac{i}{2}}\nabla^{i} \nabla
d(z,\tau)|^{2}\right)\text{d}z\text{d}\tau\nonumber\\
\leq&C_{m} 2^{q(n+1)} \sum_{i=0}^{m}\int_{0}^{2^{-2q}}
\int_{|z-x|\leq 2^{1-q}}
\left(|\tau^{\frac{i}{2}}\nabla^{i} u(z,\tau)|^{2}+|\tau^{\frac{i}{2}}\nabla^{i} \nabla d(z,\tau)|^{2}\right)\text{d}z\text{d}\tau\nonumber\\
\leq&C_{m} 2^{q(n+1)} 2^{-qn}
\sum_{i=0}^{m}\left(\|\tau^{\frac{i}{2}}\nabla^{i}
u\|_{Z}^{2}+\|\tau^{\frac{i}{2}}\nabla^{i} d\|_{X}^{2}\right)\leq
C_{m} 2^{q}\varepsilon^{2}.
\end{align*}
As a consequence, if we select $\varepsilon>0$ small enough, we have
\begin{align*}
|t^{\frac{m}{2}}\nabla^{m}\Delta_{q} \mathbb{T}_{1}(u,d)(x,t)|\leq
C_{m} 2^{q} \varepsilon^{2}\leq C_{m} 2^{q} \varepsilon \quad\text{
for } t\in (0,2^{-q}),
\end{align*}
which together with \eqref{eq3.11} implies that
\begin{align}\label{eq3.12}
\|t^{\frac{m}{2}}\nabla^{m}\Delta_{q}
\mathbb{T}_{1}(u,d)\|_{\widetilde{L}^{\infty}(\mathbb{R}_{+};\dot{B}^{-1}_{\infty,\infty})}\leq
C_{m}\varepsilon.
\end{align}
%%----------------------------(eq3.12)------------------------------------
Now let us turn to  estimate the
$\widetilde{L}^{1}(\mathbb{R}_{+};\dot{B}^{1}_{\infty,\infty})$-norm
of $t^{\frac{m}{2}}\nabla^{m}\mathbb{T}_{1}(u,d)$, by using
\eqref{eq3.9} and Lemma \ref{lem2.5} again, we have
\begin{align}\label{eq3.13}
&2^{q}\int_{0}^{2^{-2q}}\|t^{\frac{m}{2}}\nabla^{m}\Delta_{q}\mathbb{T}_{1}(u,d)(\cdot,t)\|_{L^{\infty}}\text{d}t
\nonumber\\
\leq&
2^{q}\int_{0}^{2^{-2q}}\|t^{\frac{m}{2}}\nabla^{m}u(\cdot,t)\|_{L^{\infty}}\text{d}t
+2^{q}\int_{0}^{2^{-2q}}\|t^{\frac{m}{2}}\nabla^{m}e^{t\Delta}
u_{0}\|_{L^{\infty}}\text{d}t\nonumber\\
\leq& 2^{q}\int_{0}^{2^{-2q}} t^{-\frac{1}{2}}
\text{d}t\|t^{\frac{m}{2}}\nabla^{m}u\|_{Z} + C_{m}
\|u_{0}\|_{BMO^{-1}}\leq C_{m}\varepsilon,
\end{align}
%%----------------------------(eq3.13)------------------------------------
where we have used \eqref{eq1.5} in the last above inequality. To
complete the proof, we still need to estimate  the term $ 2^{q}
\int_{2^{-2q}}^{\infty}\left\|t^{\frac{m}{2}}\nabla^{m}\Delta_{q}
\mathbb{T}_{1}(u,d)(\cdot,t)\right\|_{L^{\infty}} \text{d}t $. In
order to do it,  for all $q\in\mathbb{Z}$, we split $\Delta_{q}
\mathbb{T}_{1}(u,d)$ as
\begin{align*}
\Delta_{q}
\mathbb{T}_{1}(u,d)(x,t)=&\int_{0}^{\frac{t}{2}}\Delta_{q}
e^{(t-\tau)\Delta} \mathbb{P}\nabla \cdot
\mathbf{F}(x,\tau)\text{d}\tau+\int_{\frac{t}{2}}^{t} \int_{|y|\geq
2t^{\frac{1}{2}}} g_{q}(y,t-\tau)\cdot
\mathbf{F}(x-y,\tau)\text{d}y\text{d}\tau\nonumber\\
&+\int_{\frac{t}{2}}^{t} \int_{|y|\leq 2t^{\frac{1}{2}}}
g_{1,q}(y,t-\tau)\nabla\cdot
\mathbf{F}(x-y,\tau)\text{d}y\text{d}\tau\nonumber\\
:=&
(\widetilde{F}_{1,q}+\widetilde{F}_{2,q}+\widetilde{F}_{3,q})(x,t).
\end{align*}
By applying \eqref{eq3.9} and Lemma \ref{lem2.5} that
\begin{align*}
&2^{q}\int_{2^{-2q}}^{\infty}\|t^{\frac{m}{2}}\nabla^{m}\widetilde{F}_{1,q}(\cdot,t)\|_{L^{\infty}}\text{d}t\nonumber\\
=&2^{q}\int_{2^{-2q}}^{\infty}\left\|t^{\frac{m}{2}}\nabla^{m}\Delta_{q}e^{\frac{t}{2}\Delta}\left(u(x,\frac{t}{2})
-(e^{\frac{t}{2}\Delta}u_{0})(x)\right)\right\|_{L^{\infty}}\text{d}t\nonumber\\
\leq& 2^{q} \int_{2^{-2q}}^{\infty}\left(\left\|
t^{\frac{m}{2}}\nabla^{m}\Delta_{q}e^{\frac{t}{2}\Delta}
u(x,\frac{t}{2})\right\|_{L^{\infty}}
+\left\|t^{\frac{m}{2}}\nabla^{m}\Delta_{q}\left(e^{t\Delta}u_{0}\right)(x)\right\|_{L^{\infty}}\right)\text{d}t
\nonumber\\
\leq&C_{m} 2^{q}\int_{2^{-2q}}^{\infty}\left(
\left\|\int_{\mathbb{R}^{n}} g_{3,q}(y,\frac{t}{2})\cdot
u(x-y,\frac{t}{2})\text{d}y\right\|_{L^{\infty}}+\left\|g_{3,q}\ast
\Delta_{q} u_{0}\right\|_{L^{\infty}}\right)\text{d}t\nonumber\\
\leq& C_{m}2^{q}\int_{2^{-2q}}^{\infty}\int_{\mathbb{R}^{n}}
\frac{2^{qn}}{1+(2^{q}y)^{2n}} e^{-ct2^{2q}}\text{d}y
\left(\left\|u(\cdot,\frac{t}{2})\right\|_{L^{\infty}}+\left\|\Delta_{q}
u_{0}\right\|_{L^{\infty}}\right)\text{d}t\nonumber\\
\leq&C_{m}2^{q}\int_{2^{-2q}}^{\infty}e^{-ct2^{2q}}
t^{-\frac{1}{2}}\text{d}t
\|u\|_{Z}+C_{m}2^{q}\int_{2^{-2q}}^{\infty}e^{-ct2^{2q}}\text{d}t
\|\Delta_{q} u_{0}\|_{L^{\infty}}\nonumber\\
\leq & C_{m}\|u\|_{Z}+C_{m} 2^{-q}\|\Delta_{q}
u_{0}\|_{L^{\infty}}\leq C_{m}
(\|u\|_{Z}+\|u_{0}\|_{\dot{B}^{-1}_{\infty,\infty}})\nonumber\\
\leq& C_{m}\varepsilon+C_{m}\|u_{0}\|_{BMO^{-1}}\leq
C_{m}\varepsilon,
\end{align*}
where we have used the estimate \eqref{eq1.5} and the embedding
$BMO^{-1}\hookrightarrow \dot{B}^{-1}_{\infty,\infty}$. Exactly
following the same line, we have
\begin{align*}
&2^{q}\int_{2^{-2q}}^{\infty}\|t^{\frac{m}{2}}\nabla^{m}\widetilde{F}_{2,q}(\cdot,t)\|_{L^{\infty}}\text{d}t\nonumber\\
\leq& C_{m} 2^{q} \int_{2^{-2q}}^{\infty}\left\|
\int_{\frac{t}{2}}^{t}\int_{|y|\geq
2t^{\frac{1}{2}}}\frac{2^{q(n+1)}}{1+(2^{q}|y|)^{2n}}e^{-c(t-\tau)2^{2q}}
|t^{\frac{m}{2}}\nabla^{m} \mathbf{F}
(x-y,\tau)|\text{d}y\text{d}\tau\right\|_{L^{\infty}}\text{d}t\nonumber\\
\leq& C_{\!m}\! 2^{q} \!\!\int_{2^{-2q}}^{\infty}\!\left\|
\int_{\frac{t}{2}}^{t}\!\int_{|y|\geq
2t^{\frac{1}{2}}}\!\frac{2^{q(n+1)}}{1\!+\!(2^{q}|y|)^{n}}\!e^{\!-c(t-\tau)2^{2q}}\!
\sum_{i=0}^{m}\!\left( |t^{\frac{i}{2}}\nabla^{i}
u(x-y,\tau)|^{2}\!+|t^{\frac{i}{2}}\nabla^{i}\nabla d(x-y,\tau)|^{2}
\right)\text{d}y\text{d}\tau\right\|_{\!L^{\!\infty}}\!\text{d}t\nonumber\\
\leq& C_{\!m}\!2^{\!q}\!\!\sum_{\ell\in \mathbb{Z}^{n}\!\backslash
\{0\}}\!\!\int_{\!2^{-2q}}^{\infty}\!\left\|
\int_{\!\frac{t}{2}}^{t}\!\int_{\!
|y|\in\sqrt{t}(\ell+[0,1]^{n})}\!\frac{2^{q(n+1)}}{1\!+\!(2^{q}|y|)^{2n}}\!
\sum_{i=0}^{m}\!\left(\! |\tau^{\frac{i}{2}}\nabla^{i}
u(x\!-\!y,\tau)|^{2}\!+|\tau^{\frac{i}{2}}\nabla^{i}\nabla
d(x\!-\!y,\tau)|^{2}
\right)\!\text{d}y\text{d}\tau\right\|_{\!L^{\!\infty}}\!\!\text{d}t\nonumber\\
\leq & C_{\!m}\!2^{\!q}\!\!\sum_{\ell\in \mathbb{Z}^{n}\!\backslash
\{0\}}\!\!\int_{\!2^{-2q}}^{\infty}\!\left\|\frac{2^{q(1-n)}}{(|\ell|\sqrt{t})^{2n}}\!
\int_{\!\frac{t}{2}}^{t}\!\int_{\!
|y|\in\sqrt{t}(\ell+[0,1]^{2n})}\! \sum_{i=0}^{m}\!\left(\!
|\tau^{\frac{i}{2}}\nabla^{i}
u(x\!-\!y,\tau)|^{2}\!+|\tau^{\frac{i}{2}}\nabla^{i}\nabla
d(x\!-\!y,\tau)|^{2}
\right)\!\text{d}y\text{d}\tau\right\|_{\!L^{\!\infty}}\!\!\text{d}t\nonumber\\
\leq & C_{m}2^{q}\sum_{\ell\in \mathbb{Z}^{n}\backslash
\{0\}}\!\int_{2^{-2q}}^{\infty}\left(\frac{2^{q(1-n)}}{(|\ell|\sqrt{t})^{2n}}\!
t^{\frac{n}{2}}\sum_{i=0}^{m}\left( \|\tau^{\frac{i}{2}}\nabla^{i}
u\|_{Z}^{2}+\|\tau^{\frac{i}{2}}\nabla^{i}d\|_{X}^{2}\right)\right)\text{d}t
\nonumber\\
\leq & C_{m}\sum_{\ell\in \mathbb{Z}^{n}\backslash
\{0\}}\frac{1}{|\ell|^{2n}}\sum_{i=0}^{m}\left(
\|\tau^{\frac{i}{2}}\nabla^{i}
u\|_{Z}^{2}+\|\tau^{\frac{i}{2}}\nabla^{i}d\|_{X}^{2}\right)\leq
C_{m}\varepsilon^{2},
\end{align*}
and
\begin{align*}
&2^{q}\int_{2^{-2q}}^{\infty}\|t^{\frac{m}{2}}\nabla^{m}\widetilde{F}_{3,q}(\cdot,t)\|_{L^{\infty}}\text{d}t\nonumber\\
\leq& C_{m} 2^{q} \int_{2^{-2q}}^{\infty}\left\|
\int_{\frac{t}{2}}^{t}\int_{|y|\leq
2t^{\frac{1}{2}}}\frac{2^{qn}}{1+(2^{q}|y|)^{2n}}e^{-c(t-\tau)2^{2q}}
|t^{\frac{m}{2}}\nabla^{m+1} \mathbf{F}
(x-y,\tau)|\text{d}y\text{d}\tau\right\|_{L^{\infty}}\text{d}t\nonumber\\
\leq& C_{\!m}\! 2^{q} \!\!\int_{2^{-2q}}^{\infty}\!\left\|
\int_{\frac{t}{2}}^{t}\!\!\int_{\!|y|\leq
2t^{\frac{1}{2}}}\!\!\frac{2^{qn}}{1\!+\!(2^{q}|y|)^{2n}}\!e^{\!-c(t-\tau)2^{2q}}\!\frac{1}{\sqrt{t}}\!
\sum_{i=0}^{m\!+\!1}\!\!\left( |t^{\frac{i}{2}}\nabla^{i}
u(x\!-\!y,\tau)|^{2}\!\!+|t^{\frac{i}{2}}\nabla^{i}\nabla
d(x\!-\!y,\tau)|^{2}
\right)\!\text{d}y\text{d}\tau\right\|_{\!L^{\!\infty}}\!\!\text{d}t\nonumber\\
\leq& C_{\!m}\! 2^{q} \!\!\int_{2^{-2q}}^{\infty}\!\left\|
\int_{\frac{t}{2}}^{t}\!\!\int_{\!|y|\leq
2t^{\frac{1}{2}}}\!\frac{2^{qn}}{1\!+\!(2^{q}|y|)^{2n}}\!e^{\!-c(t-\tau)2^{2q}}\!\frac{1}{\sqrt{t}}\!
\sum_{i=0}^{m\!+\!1}\!\!\left( |\tau^{\frac{i}{2}}\nabla^{i}
u(x\!-\!y,\tau)|^{2}\!\!+\!|\tau^{\frac{i}{2}}\nabla^{i}\nabla
d(x\!-\!y,\tau)|^{2}
\right)\!\text{d}y\text{d}\tau\right\|_{\!L^{\!\infty}}\!\!\text{d}t\nonumber\\
\leq& C_{\!m}\! 2^{q} \!\!\int_{2^{-2q}}^{\infty}\!\left\|
\int_{\frac{t}{2}}^{t}\!\int_{|y|\leq
2t^{\frac{1}{2}}}\!\frac{2^{qn}}{1\!+\!(2^{q}|y|)^{2n}}\!e^{\!-c(t-\tau)2^{2q}}\!
\tau^{-\frac{3}{2}}\sum_{i=0}^{m+1}\!\left(
\|t^{\frac{i}{2}}\nabla^{i}
u\|_{Z}^{2}\!+\|t^{\frac{i}{2}}\nabla^{i}d\|_{X}^{2}
\right)\text{d}y\text{d}\tau\right\|_{\!L^{\!\infty}}\!\text{d}t\nonumber\\
\leq& C_{m} \varepsilon^{2}
2^{q}\int_{2^{-2q}}^{\infty}\int_{\frac{t}{2}}^{t}
e^{-c(t-\tau)2^{2q}} \tau^{-\frac{3}{2}}\text{d}\tau\text{d}t\leq
C_{m}\varepsilon^{2}.
\end{align*}
Therefore, we get
\begin{align*}
2^{q}\int_{2^{-2q}}^{\infty}\|t^{\frac{m}{2}}\nabla^{m}\Delta_{q}\mathbb{T}_{1}(u,d)(\cdot,t)\|_{L^{\infty}}\text{d}t\leq
C_{m}\varepsilon^{2},
\end{align*}
which along with \eqref{eq3.13} ensures that
\begin{align*}
\|t^{\frac{m}{2}}\nabla^{m}\Delta_{q}
\mathbb{T}_{1}(u,d)\|_{\widetilde{L}^{1}(\mathbb{R}_{+};\dot{B}^{1}_{\infty,\infty})}\leq
C_{m}\varepsilon(1+\varepsilon)\leq C_{m}\varepsilon,
\end{align*}
if we choose $\varepsilon$ small enough. The above inequality
together with \eqref{eq3.12} gives the needed estimates of $u$.

To complete the proof of \eqref{eq3.10}, it remains to prove that
for any positive integers $m\geq 0$, there holds
\begin{align}\label{eq3.14}
\|t^{\frac{m}{2}} \nabla^{m+1}
\mathbb{T}_{2}(u,d)(x,t)\|_{\widetilde{L}^{\infty}(\mathbb{R}_{+};\dot{B}^{-1}_{\infty,\infty})\cap
\widetilde{L}^{1}(\mathbb{R}_{+};\dot{B}^{1}_{\infty,\infty})}\leq
C_{m}\varepsilon.
\end{align}
%%----------------------------(eq3.14)------------------------------------
 For $t\in
[2^{-2q},\infty)$, by using \eqref{eq3.9} and Lemma \ref{lem2.5}, we
have
\begin{align*}
2^{-q}\|\Delta_{q}t^{\frac{m}{2}} \nabla^{m+1}
\mathbb{T}_{2}(u,d)(\cdot,t)\|_{L^{\infty}}\leq&
2^{-q}\|t^{\frac{m}{2}} \nabla^{m+1}
\mathbb{T}_{2}(u,d)(\cdot,t)\|_{L^{\infty}}\nonumber\\
=& 2^{-q}\|t^{\frac{m}{2}} \nabla^{m} \nabla\left(
d(\cdot,t)-(e^{t\Delta} d_{0})(x)\right)\|_{L^{\infty}}\nonumber\\
 \leq&
2^{-q}\|t^{\frac{m}{2}} \nabla^{m} \nabla d(\cdot,t)\|_{L^{\infty}}+
2^{-q} \|t^{\frac{m}{2}} \nabla^{m+1} e^{t\Delta}
d_{0}\|_{L^{\infty}}\nonumber\\
\leq& 2^{-q} t^{-\frac{1}{2}} \|t^{\frac{m}{2}} \nabla^{m} d\|_{X}
+C_{m}[d_{0}]_{BMO}\nonumber\\
\leq& C_{m}[d_{0}]_{BMO}\leq C_{m}\varepsilon,
\end{align*}
where we have used \eqref{eq1.5} in the last inequality above. For
$t\in (0,2^{-2q})$, by denoting
\begin{align*}
\mathbf{G}(x,t):= (-u\cdot \nabla d +|\nabla d|^{2}d)(x,t),
\end{align*}
we can rewrite $\Delta_{q} \mathbb{T}_{2}(u,d)$ for any $q\in
\mathbb{Z}$ as
\begin{align*}
\Delta_{q} \mathbb{T}_{2}(u,d) (x,t) =&\int_{0}^{\frac{t}{2}}
\int_{|y|\geq 2^{1-q}} g_{2,q}(y,t-\tau)
\mathbf{G}(x-y,\tau)\text{d}y\text{d}\tau\nonumber\\
&+\int_{\frac{t}{2}}^{t} \int_{|y|\geq 2^{1-q}} g_{2,q}(y,t-\tau)
\mathbf{G}(x-y,\tau)\text{d}y\text{d}\tau\nonumber\\
&+ \int_{0}^{\frac{t}{2}} \int_{|y|\leq 2^{1-q}}g_{2,q}(y,t-\tau)
\mathbf{G}(x-y,\tau)\text{d}y\text{d}\tau\nonumber\\
&+ \int_{\frac{t}{2}}^{t} \int_{|y|\leq 2^{1-q}}g_{2,q}(y,t-\tau)
\mathbf{G}(x-y,\tau)\text{d}y\text{d}\tau\nonumber\\
:\triangleq&( G_{1,q}+G_{2,q}+G_{3,q}+G_{4,q})(x,t),
\end{align*}
where  $g_{2,q}$ is given by Lemma \ref{lem2.5}. We shall estimate
term by term from $G_{1,q}$ to $G_{3,q}$. Indeed thanks to
\eqref{eq1.5} and Lemma \ref{lem2.5}, we have
\begin{align*}
|t^{\frac{m}{2}}\nabla^{m+1}G_{1,q}(x,t)|
=&\int_{0}^{\frac{t}{2}}\int_{|y|\geq 2^{1-q}}
\left(\frac{t}{t-\tau}\right)^{\frac{m}{2}}(t-\tau)^{\frac{m}{2}}
(\nabla^{m+1} g_{2,q}) (y,t-\tau)
\mathbf{G}(x-y,\tau)\text{d}y\text{d}\tau\nonumber\\
\leq& C_{m}\!\! \int_{0}^{\frac{t}{2}} \!\!\int_{|y|\geq 2^{1-q}}\!
\frac{2^{qn}}{1+(2^{q}|y|)^{2n}} e^{-c(t-\tau)2^{2q}}
(t-\tau)^{-\frac{1}{2}}
|\mathbf{G}(x-y,\tau)|\text{d}y\text{d}\tau\nonumber\\
\leq& C_{m}\!\! \int_{0}^{\frac{t}{2}} \!\!\int_{|y|\geq 2^{1-q}}\!
\frac{2^{qn}}{1+(2^{q}|y|)^{2n}} e^{c\tau2^{2q}} \tau^{-\frac{1}{2}}
|\mathbf{G}(x-y,\tau)|\text{d}y\text{d}\tau\nonumber\\
\leq&C_{m}\!\! \sum_{\ell\in \mathbb{Z}^{n}\backslash\{0\}}
\int_{0}^{t} \int_{y\in
2^{-q}(\ell+[0,1]^{n})}\frac{2^{q(n+1)}}{1+(2^{q}|y|)^{2n}}
|\mathbf{G}(x-y,\tau)|\text{d}y\text{d}\tau\nonumber\\
\leq&C_{m}\!\! \sum_{\ell\in
\mathbb{Z}^{n}\backslash\{0\}}\!\!\frac{2^{qn}}{|\ell|^{2n}}\!2^{q}\int_{0}^{2^{-2q}}\!
\int_{y\in 2^{-q}(\ell+[0,1]^{n})}\!\left(|u(x-y,\tau)|^{2}+|\nabla
d(x-y,\tau)|^{2}\right)\text{d}y\text{d}\tau\nonumber\\
\leq& C_{m}2^{q } \left(\|u\|_{Z}^{2}+\|d\|_{X}^{2}\right)\leq C_{m}
2^{q} \varepsilon^{2},
\end{align*}
where we have used the fact that $|d|=1$ and $e^{c\tau2^{2q}}
\tau^{-\frac{1}{2}}\leq C 2^{q}$ for all $0<\tau<\infty$. Exactly
following the same line, we have
\begin{align*}
&|t^{\frac{m}{2}}\nabla^{m+1}G_{2,q}(x,t)|\nonumber\\
= &\int_{\frac{t}{2}}^{t}\int_{|y|\geq
2^{1-q}}g_{2,q}(y,t-\tau)t^{\frac{m}{2}} \nabla^{m+1}
\mathbf{G}(x-y,\tau)\text{d}y\text{d}\tau\nonumber\\
\leq&\!C_{\!m}\! \int_{\!\frac{t}{2}}^{t}\!\! \int_{\!|y|\geq
2^{1\!-\!q}}\!\! \frac{2^{qn}}{1\!+\!(2^{q}y)^{2n}}
e^{-c(t\!-\!\tau)2^{2q}}\!\frac{1}{\sqrt{t}} \!\sum_{i=0}^{m+1}\!\!
\Big(|t^{\frac{i}{2}}\nabla^{i}
u(x\!-\!y,\tau)|^{2}\!\!+\!|t^{\frac{i}{2}}\nabla^{i} \nabla
d(x\!-\!y,\tau)|^{2}|t^{\frac{i}{2}}\nabla^{i}
d(x\!-\!y,\tau)|\Big)\!\text{d}y\text{d}\tau\nonumber\\
\leq&\!C_{\!m}\! \int_{\!\frac{t}{2}}^{t}\!\! \int_{\!|y|\geq
2^{1\!-\!q}}\!\! \frac{2^{q(n+1)}}{1\!+\!(2^{q}y)^{2n}}
\!\sum_{i=0}^{m+1}\!\! \Big(|\tau^{\frac{i}{2}}\nabla^{i}
u(x\!-\!y,\tau)|^{2}\!\!+\!|\tau^{\frac{i}{2}}\nabla^{i} \nabla
d(x\!-\!y,\tau)|^{2}|\tau^{\frac{i}{2}}\nabla^{i}
d(x\!-\!y,\tau)|\Big)\!\text{d}y\text{d}\tau\nonumber\\
\leq&C_{m} \!\!\sum_{\ell\in\mathbb{Z}^{n}\backslash
\{0\}}\!\!\!\frac{2^{q(n\!+\!1\!)\!}}{|\ell|^{2n}}\!
\!\left(\!1\!+\!\!\sum_{i=1}^{m\!+\!1}\!\|t^{\frac{i}{2}}\!\nabla^{i}\!
d(\cdot,t)\|_{\!L^{\!\infty}}\!\!\right)\!\int_{0}^{2^{-\!2q}}\!\!\!\!
\int_{\!y\in 2^{\!-\!q}(\ell\!+\![0,1]^{n}\!)}
\!\!\sum_{i=0}^{m+1}\! \Big(|\tau^{\frac{i}{2}}\!\nabla^{i}
u(x\!-\!y,\tau)|^{2}\!\!+\!|\tau^{\frac{i}{2}}\!\nabla^{i} \nabla
d(x\!-\!y,\tau)|^{2}\!\Big)\!\text{d}y\text{d}\tau\nonumber\\
\leq& C_{m}
2^{q(n+1)}\left(1+\sum_{i=0}^{m}\|t^{\frac{i}{2}}\nabla^{i}
d\|_{X}\right)\cdot 2^{-qn}\sum_{i=0}^{m+1}
\left(\|\tau^{\frac{i}{2}}\nabla^{i}u\|_{Z}^{2}+\|\tau^{\frac{i}{2}}\nabla^{i}
d\|_{X}^{2}\right)\nonumber\\
\leq& C_{m} 2^{q} \varepsilon^{2}(1+\varepsilon);\nonumber\\
  %%%%%%%%%%%%%%%%%%%%%%%%%%%%%%%%%%%%%%%%%%%%%%%%%%%%%%%%%%%%%%%%%%%%%%%%%%%%%%%%%%%%%%%%%%%%%5555
&|t^{\frac{m}{2}}\nabla^{m+1}G_{3,q}(x,t)|\nonumber\\
=&\int_{0}^{\frac{t}{2}}\int_{|y|\leq
2^{1-q}}\left(\frac{t}{t-\tau}\right)^{\frac{m}{2}}(t-\tau)^{\frac{m}{2}}
(\nabla^{m+1} g_{2,q})(y,t-\tau)
\mathbf{G}(x-y,\tau)\text{d}y\text{d}\tau\nonumber\\
\leq& C_{m}\int_{0}^{\frac{t}{2}} \int_{|y|\leq
2^{1-q}}\frac{2^{qn}}{1+(2^{q} |y|)^{2n}} e^{-c(t-\tau)2^{2q}}
\frac{1}{\sqrt{t-\tau}}
|\mathbf{G}(x-y,\tau)|\text{d}y\text{d}\tau\nonumber\\
\leq& C_{m}\int_{0}^{\frac{t}{2}} \int_{|y|\leq
2^{1-q}}\frac{2^{q(n+1)}}{1+(2^{q} |y|)^{2n}}
|\mathbf{G}(x-y,\tau)|\text{d}y\text{d}\tau\nonumber\\
\leq& C_{m} 2^{q(n+1)} \int_{0}^{t} \int_{|z-x|\leq 2^{1-q}}
|\mathbf{G}(z,\tau)|\text{d}z\text{d}\tau\nonumber\\
\leq&C_{m} 2^{q(n+1)} \int_{0}^{2^{-2q}} \int_{|z-x|\leq 2^{1-q}}
\left(|u(z,\tau)|^{2}+|\nabla d(z,\tau)|^{2}\right)\text{d}z\text{d}\tau\nonumber\\
\leq&C_{m} 2^{q(n+1)} 2^{-qn}
\left(\|u\|_{Z}^{2}+\|d\|_{X}^{2}\right)\leq C_{m}
2^{q}\varepsilon^{2},
\end{align*}
and
\begin{align*}
&|t^{\frac{m}{2}}\nabla^{m+1}G_{4,q}(x,t)|
=\int_{\frac{t}{2}}^{t}\int_{|y|\leq 2^{1-q}}  g_{2,q}(y,t-\tau)
t^{\frac{m}{2}} \nabla^{m+1}\mathbf{G}(x-y,\tau)\text{d}y\text{d}\tau\nonumber\\
\leq&\!C_{\!m}\!\! \int_{\!\frac{t}{2}}^{t}\!\! \int_{\!|y|\leq
2^{1\!- \!q}} \!\frac{2^{qn}}{1\!+\!(2^{q}y\!)^{2n}}
e^{-c(t\!-\!\tau)2^{2q}}\!\!\frac{1}{\sqrt{t}} \!\sum_{i=0}^{m+1}\!
\Big(|t^{\frac{i}{2}}\!\nabla^{i}
u(x\!-\!y,\tau)|^{2}\!\!+\!|t^{\frac{i}{2}}\!\nabla^{i} \nabla
d(x\!-\!y,\tau)|^{2}|t^{\frac{i}{2}}\!\nabla^{i}
d(x\!-\!y,\tau)|\!\Big)\!\text{d}y\text{d}\tau\nonumber\\
\leq&C_{m} \int_{\frac{t}{2}}^{t} \int_{|y|\leq 2^{1-q}}
\frac{2^{q(n+1)}}{1+(2^{q}y)^{2n}}
 \!\sum_{i=0}^{m+1}\!
\Big(|\tau^{\frac{i}{2}}\nabla^{i} u(x\!-\!y,\tau)|^{2}\!\!
+\!|\tau^{\frac{i}{2}}\nabla^{i} \nabla
d(x\!-\!y,\tau)|^{2}|\tau^{\frac{i}{2}}\nabla^{i}
d(x\!-\!y,\tau)|\Big)\!\text{d}y\text{d}\tau\nonumber\\
\leq&C_{m} 2^{q(n+1)} \!\sum_{i=0}^{m+1}\!\int_{0}^{\!2^{-2q}}\!\!\!
\int_{|x-z|\leq 2^{1-q}} \!\!\left(|\tau^{\frac{i}{2}}\nabla^{i}
u(z,\tau)|^{2}\!\!+\!|\tau^{\frac{i}{2}}\nabla^{i} \nabla
d(z,\tau)|^{2}\!\right)\!\text{d}z\text{d}\tau\cdot
\left(1+\!\sum_{i=1}^{m+1}\|t^{\frac{i}{2}}\nabla^{i}d(\cdot,t)\|_{L^{\infty}}\right)\nonumber\\
\leq&C_{m}
2^{q}\left(1+\!\sum_{i=0}^{m}\|t^{\frac{i}{2}}\nabla^{i}d\|_{Z}\right)\sum_{i=0}^{m+1}
\left(\|\tau^{\frac{i}{2}}\nabla^{i}u\|_{Z}^{2}+\|\tau^{\frac{i}{2}}\nabla^{i}
d\|_{X}^{2}\right)\leq C_{m} 2^{q} \varepsilon^{2}(1+\varepsilon),
\end{align*}
where we have used the fact that $|d|=1$ in the above inequalities.
 As a consequence, we obtain
\begin{align}\label{eq3.15}
\|t^{\frac{m}{2}} \nabla^{m+1}
\mathbb{T}_{2}(u,d)(x,t)\|_{\widetilde{L}^{\infty}(\mathbb{R}_{+};\dot{B}^{-1}_{\infty,\infty})}\leq
C_{m}\varepsilon(1+\varepsilon(1+\varepsilon))\leq C_{m}\varepsilon,
\end{align}
%%----------------------------(eq3.15)------------------------------------
if we selecting $\varepsilon$ small enough. On the other hand, we
have
\begin{align}\label{eq3.16}
&2^{q}\int_{0}^{2^{-2q}}\|t^{\frac{m}{2}}\nabla^{m+1}\Delta_{q}\mathbb{T}_{2}(u,d)(\cdot,t)\|_{L^{\infty}}\text{d}t
\nonumber\\
\leq& 2^{q}\int_{0}^{2^{-2q}}\|t^{\frac{m}{2}}\nabla^{m}\nabla
d(\cdot,t)\|_{L^{\infty}}\text{d}t
+2^{q}\int_{0}^{2^{-2q}}\|t^{\frac{m}{2}}\nabla^{m+1}e^{t\Delta}
d_{0}\|_{L^{\infty}}\text{d}t\nonumber\\
\leq& 2^{q}\int_{0}^{2^{-2q}} t^{-\frac{1}{2}}
\text{d}t\|t^{\frac{m}{2}}\nabla^{m}d\|_{X} + C_{m}
[d_{0}]_{BMO}\leq C_{m}\varepsilon,
\end{align}
%%----------------------------(eq3.16)------------------------------------
Now let us turn to the estimate of $\int_{2^{-2q}}^{\infty}
\|t^{\frac{m}{2}}\nabla^{m+1}\Delta_{q}\mathbb{T}_{2}(u,d)(\cdot,t)\|_{L^{\infty}}\text{d}t$.
We split $\Delta_{q}\mathbb{T}_{2}(u,d)$ for all $q\in\mathbb{Z}$ as
\begin{align*}
\Delta_{q}\mathbb{T}_{1}(u,d)(x,t)=&\int_{0}^{\frac{t}{2}}
\Delta_{q} e^{(t-\tau)\Delta}
\mathbf{G}(x,\tau)\text{d}\tau+\int_{\frac{t}{2}}^{t}\int_{|y|\geq
2\sqrt{t}}
g_{2,q}(y,t-\tau)\mathbf{G}(x-y,\tau)\text{d}y\text{d}\tau\nonumber\\
&+\int_{\frac{t}{2}}^{t}\int_{|y|\leq 2\sqrt{t}}
g_{2,q}(y,t-\tau)\mathbf{G}(x-y,\tau)\text{d}y\text{d}\tau\nonumber\\
:\triangleq&
(\widetilde{G}_{1,q}+\widetilde{G}_{2,q}+\widetilde{G}_{3,q})(x,t).
\end{align*}
We  shall estimate term by term from $\widetilde{G}_{1,q}$ to
$\widetilde{G}_{3,q}$. Indeed, thanks to \eqref{eq3.9} and Lemma
\ref{lem2.5}, we have
\begin{align*}
&2^{q}\int_{2^{-2q}}^{\infty}\|t^{\frac{m}{2}}\nabla^{m+1}\widetilde{G}_{1,q}(\cdot,t)\|_{L^{\infty}}\text{d}t\nonumber\\
=&2^{q}\int_{2^{-2q}}^{\infty}\left\|t^{\frac{m}{2}}\nabla^{m+1}\Delta_{q}e^{\frac{t}{2}\Delta}\left(d(x,\frac{t}{2})
-(e^{\frac{t}{2}\Delta}d_{0})(x)\right)\right\|_{L^{\infty}}\text{d}t\nonumber\\
\leq& 2^{q} \int_{2^{-2q}}^{\infty}\left(\left\|
t^{\frac{m}{2}}\nabla^{m+1}\Delta_{q}e^{\frac{t}{2}\Delta}
d(x,\frac{t}{2})\right\|_{L^{\infty}}
+\left\|t^{\frac{m}{2}}\nabla^{m+1}\Delta_{q}\left(e^{t\Delta}d_{0}\right)(x)\right\|_{L^{\infty}}\right)\text{d}t
\nonumber\\
\leq&C_{m} 2^{q}\int_{2^{-2q}}^{\infty}\left(
\left\|\int_{\mathbb{R}^{n}} g_{3,q}(y,\frac{t}{2})\cdot \nabla
d(x-y,\frac{t}{2})\text{d}y\right\|_{L^{\infty}}+\left\|g_{3,q}\ast
\Delta_{q} \nabla d_{0}\right\|_{L^{\infty}}\right)\text{d}t\nonumber\\
\leq& C_{m}2^{q}\int_{2^{-2q}}^{\infty}\int_{\mathbb{R}^{n}}
\frac{2^{qn}}{1+(2^{q}y)^{2n}} e^{-ct2^{2q}}\text{d}y
\left(\left\|\nabla
d(\cdot,\frac{t}{2})\right\|_{L^{\infty}}+2^{q}\left\|\Delta_{q}
d_{0}\right\|_{L^{\infty}}\right)\text{d}t\nonumber\\
\leq&C_{m}2^{q}\int_{2^{-2q}}^{\infty}e^{-ct2^{2q}}
t^{-\frac{1}{2}}\text{d}t
\|d\|_{X}+C_{m}2^{2q}\int_{2^{-2q}}^{\infty}e^{-ct2^{2q}}\text{d}t
\|\Delta_{q} d_{0}\|_{L^{\infty}}\nonumber\\
\leq & C_{m}\|d\|_{X}+C_{m} \|\Delta_{q} d_{0}\|_{L^{\infty}}\leq
C_{m}
(\|d\|_{X}+\|d_{0}\|_{\dot{B}^{0}_{\infty,\infty}})\nonumber\\
\leq& C_{m}\varepsilon+C_{m}[u_{0}]_{BMO}\leq C_{m}\varepsilon,
\end{align*}
where we have used the estimate \eqref{eq1.5} and the embedding
$BMO\hookrightarrow \dot{B}^{0}_{\infty,\infty}$. Exactly following
the same line, we have
\begin{align*}
&2^{q}\int_{2^{-2q}}^{\infty}\|t^{\frac{m}{2}}\nabla^{m+1}\widetilde{G}_{2,q}(\cdot,t)\|_{L^{\infty}}\text{d}t\nonumber\\
\leq& C_{m} 2^{q} \int_{2^{-2q}}^{\infty}\Big\|
\int_{\frac{t}{2}}^{t}\int_{|y|\geq
2t^{\frac{1}{2}}}\frac{2^{qn}}{1+(2^{q}|y|)^{2n}}e^{-c(t-\tau)2^{2q}}
|t^{\frac{m}{2}}\nabla^{m+1} \mathbf{G}
(x-y,\tau)|\text{d}y\text{d}\tau\Big\|_{L^{\infty}}\text{d}t\nonumber\\
\leq& \!C_{\!m}\! 2^{q} \!\!\!\int_{2^{-\!2q}}^{\infty}\!\Big\|
\int_{\!\frac{t}{2}}^{t}\!\int_{\!|y|\geq
2t^{\frac{1}{2}}}\!\frac{2^{qn}}{1\!+\!(2^{q}|y|\!)^{2n}}\!\frac{1}{\sqrt{t}}\!
\sum_{i=0}^{m\!+\!1}\!\left( |t^{\frac{i}{2}}\nabla^{i}
u(x\!-\!y\!,\tau)|^{2}\!\!+\!|t^{\frac{i}{2}}\!\nabla^{i}\nabla
d(x\!-\!y\!,\tau)|^{2}|t^{\frac{i}{2}}\!\nabla^{i}
d(x\!-\!y\!,\tau)|\!
\right)\!\text{d}y\text{d}\tau\Big\|_{\!L^{\!\infty}}\!\!\text{d}t\nonumber\\
\leq& C_{\!m}\!2^{\!q}\!\!\sum_{\ell\in \mathbb{Z}^{n}\!\backslash
\{0\}}\!\!\int_{\!2^{-2q}}^{\infty}\!\Big\|
\int_{\!\frac{t}{2}}^{t}\!\int_{\!
|y|\in\sqrt{t}(\ell+[0,1]^{n})}\!\frac{2^{qn}}{1\!+\!(2^{q}|y|)^{2n}}\!\frac{1}{\sqrt{t}}\nonumber\\
&\qquad \qquad \qquad \times\sum_{i=0}^{m\!+\!1}\left(
|\tau^{\frac{i}{2}}\nabla^{i}
u(x-y,\tau)|^{2}+|\tau^{\frac{i}{2}}\nabla^{i}\nabla
d(x-y,\tau)|^{2} |\tau^{\frac{i}{2}}\nabla^{i} d(x-y,\tau)|
\right)\!\text{d}y\text{d}\tau\Big\|_{\!L^{\!\infty}}\!\!\text{d}t\nonumber\\
\leq & C_{\!m}\!2^{\!q}\!\!\sum_{\ell\in \mathbb{Z}^{n}\!\backslash
\{0\}}\!\!\int_{\!2^{-2q}}^{\infty}\Big\|\frac{2^{-qn}}{(|\ell|\sqrt{t})^{2n}}\!\frac{1}{\sqrt{t}}
\left(1+\sum_{i=1}^{m+1} \|t^{\frac{i}{2}}\nabla^{i} d(\cdot,t)\|_{L^{\infty}}\right)\nonumber\\
&\qquad\qquad\qquad \times\int_{\!\frac{t}{2}}^{t}\int_{
|y|\in\sqrt{t}(\ell+[0,1]^{n})} \sum_{i=0}^{m+1}\left(\!
|\tau^{\frac{i}{2}}\nabla^{i}
u(x-y,\tau)|^{2}\!+|\tau^{\frac{i}{2}}\nabla^{i}\nabla
d(x-y,\tau)|^{2}
\right)\!\text{d}y\text{d}\tau\Big\|_{\!L^{\!\infty}}\!\!\text{d}t\nonumber\\
\leq & C_{m}2^{q}\sum_{\ell\in \mathbb{Z}^{n}\backslash
\{0\}}\!\int_{2^{-2q}}^{\infty}\left(\frac{2^{-qn}}{(|\ell|\sqrt{t})^{2n}}\!
t^{\frac{n-1}{2}}\left(1+\sum_{i=0}^{m} \|t^{\frac{i}{2}}\nabla^{i}
d\|_{X}\right)\sum_{i=0}^{m+1}\left( \|\tau^{\frac{i}{2}}\nabla^{i}
u\|_{Z}^{2}+\|\tau^{\frac{i}{2}}\nabla^{i}d\|_{X}^{2}\right)\right)\text{d}t
\nonumber\\
\leq & C_{m}\sum_{\ell\in \mathbb{Z}^{n}\backslash
\{0\}}\frac{1}{|\ell|^{2n}}\left(1+\sum_{i=0}^{m}
\|t^{\frac{i}{2}}\nabla^{i} d\|_{X}\right)\sum_{i=0}^{m+1}\left(
\|\tau^{\frac{i}{2}}\nabla^{i}
u\|_{Z}^{2}+\|\tau^{\frac{i}{2}}\nabla^{i}d\|_{X}^{2}\right) \leq
C_{m}\varepsilon^{2}(1+\varepsilon),
\end{align*}
and
\begin{align*}
&2^{q}\int_{2^{-2q}}^{\infty}\|t^{\frac{m}{2}}\nabla^{m+1}\widetilde{G}_{3,q}(\cdot,t)\|_{L^{\infty}}\text{d}t\nonumber\\
\leq& C_{m} 2^{q} \int_{2^{-2q}}^{\infty}\left\|
\int_{\frac{t}{2}}^{t}\int_{|y|\leq
2t^{\frac{1}{2}}}\frac{2^{qn}}{1+(2^{q}|y|)^{2n}}e^{-c(t-\tau)2^{2q}}
|t^{\frac{m}{2}}\nabla^{m+1} \mathbf{G}
(x-y,\tau)|\text{d}y\text{d}\tau\right\|_{L^{\infty}}\text{d}t\nonumber\\
\leq& C_{\!m}\! 2^{q} \!\!\int_{2^{-2q}}^{\infty}\!\left\|
\int_{\frac{t}{2}}^{t}\!\int_{|y|\leq
2t^{\frac{1}{2}}}\!\frac{2^{qn}}{1\!+\!(2^{q}|y|)^{2n}}\!\frac{1}{\sqrt{t}}\!
\sum_{i=0}^{m\!+\!1}\!\left( |t^{\frac{i}{2}}\!\nabla^{i}
u(x\!-\!y,\tau\!)|^{2}\!\!+\!|t^{\frac{i}{2}}\!\nabla^{i}\nabla
d(x\!-\!y,\tau\!)|^{2}|t^{\frac{i}{2}}\!\nabla^{i}
d(x\!-\!y,\tau\!)|\!
\right)\!\text{d}y\text{d}\tau\right\|_{\!L^{\!\infty}}\!\text{d}t\nonumber\\
\leq& C_{\!m}\! 2^{q} \!\!\int_{2^{-2q}}^{\infty}\!\!\left\|
\int_{\!\frac{t}{2}}^{t}\!\int_{\!|y|\leq
2t^{\frac{1}{2}}}\!\frac{2^{qn}}{1\!+\!(2^{q}|y|)^{2n}}\!\frac{1}{\sqrt{t}}\!
\sum_{i=0}^{m\!+\!1}\!\left( |\tau^{\frac{i}{2}}\!\nabla^{i}
u(x\!-\!y,\tau\!)|^{2}\!\!+\!|\tau^{\frac{i}{2}}\!\nabla^{i}\nabla
d(x\!-\!y,\tau\!)|^{2}|\tau^{\frac{i}{2}}\!\nabla^{i}
d(x\!-\!y,\tau\!)|\!
\right)\!\text{d}y\text{d}\tau\right\|_{\!L^{\!\infty}}\!\text{d}t\nonumber\\
\leq& C_{\!m}\!2^{\!q}\!\!\int_{\!2^{-2q}}^{\infty}\!\left\|
\int_{\!\frac{t}{2}}^{t}\!\int_{\! |y|\leq
2\!\sqrt{t}}\!\frac{2^{qn}}{1\!\!+\!\!(2^{q}|y|)^{2n}\!}\frac{1}{\sqrt{\tau}}\!\!\sum_{i=0}^{m\!+\!1}\!\left(\!1\!+\!\|\tau^{\frac{i}{2}}\!\nabla^{i}
\!d(\!\cdot,\!\tau\!)\|_{\!L^{\!\infty}}\!\right)\!\left(\!
|\tau^{\frac{i}{2}}\nabla^{i}
u(x\!\!-\!\!y,\!\tau\!)|^{2}\!\!+\!|\tau^{\frac{i}{2}}\nabla^{i}\nabla
d(x\!\!-\!\!y,\!\tau\!)|^{2}
\right)\!\text{d}y\text{d}\tau\right\|_{\!L^{\!\infty}}\!\!\text{d}t\nonumber\\
\leq& C_{\!m}\!2^{\!q}\!\int_{\!2^{-2q}}^{\infty}\!\left\|
\int_{\!\frac{t}{2}}^{t}\!\int_{\! |y|\leq
2\sqrt{t}}\!\frac{2^{qn}}{1\!+\!(2^{q}|y|)^{2n}}\!\tau^{-\frac{3}{2}}\!
\sum_{i=0}^{m\!+\!1}\!\left(1+\|\tau^{\frac{i}{2}}\nabla^{i}
d\|_{X}\right)\!\left(\! \|\tau^{\frac{i}{2}}\nabla^{i}
u\|_{Z}^{2}+\|\tau^{\frac{i}{2}}\nabla^{i}d\|_{X}^{2}\right)\!\text{d}y\text{d}\tau\right\|_{\!L^{\!\infty}}\!\!\text{d}t\nonumber\\
\leq & C_{m}2^{q}\int_{2^{-2q}}^{\infty}\int_{\frac{t}{2}}^{t}
e^{-c(t-\tau)2^{2q}}\tau^{-\frac{3}{2}}\text{d}\tau\text{d}t\sum_{i=0}^{m+1}\!\left(1+\|\tau^{\frac{i}{2}}\nabla^{i}
d\|_{X}\right)\left( \|\tau^{\frac{i}{2}}\nabla^{i}
u\|_{Z}^{2}+\|\tau^{\frac{i}{2}}\nabla^{i}d\|_{X}^{2}\right)
\nonumber\\
\leq & C_{m}\sum_{i=0}^{m+1}\!\left(1+\|\tau^{\frac{i}{2}}\nabla^{i}
d\|_{X}\right)\left( \|\tau^{\frac{i}{2}}\nabla^{i}
u\|_{Z}^{2}+\|\tau^{\frac{i}{2}}\nabla^{i}d\|_{X}^{2}\right)\leq
C_{m}\varepsilon^{2}(1+\varepsilon).
\end{align*}
Therefore, by selecting $\varepsilon$ small enough, we obtain
\begin{align*}
2^{q}\int_{2^{-2q}}^{\infty}\|t^{\frac{m}{2}}\nabla^{m+1}\Delta_{q}
\mathbb{T}_{2}(u,d)(\cdot,t)\|_{L^{\infty}}\text{d}t\leq
C_{m}\varepsilon(1+\varepsilon(1+\varepsilon))\leq C_{m}\varepsilon,
\end{align*}
which along with  \eqref{eq3.16} ensure that
\begin{align*}
\|t^{\frac{m}{2}}\nabla^{m+1}\Delta_{q}
\mathbb{T}_{2}(u,d)\|_{\widetilde{L}^{1}(\mathbb{R}_{+};\dot{B}^{1}_{\infty,\infty})}\leq
C_{m}\varepsilon.
\end{align*}
 Combining the inequality above and \eqref{eq3.15} together, we
complete the proof of \eqref{eq3.14}. \medskip
\\
\textbf{Step 2. }  Estimate \eqref{eq1.6} for all  $k,m\geq 0$. %% i.e., we shall prove
%%\begin{align}
%%\|t^{\frac{m}{2}+k}(\partial^{k}_{t}\nabla^{m} u,
%%\partial^{k}_{t}\nabla^{m+1}
%%d)\|_{\widetilde{L}^{\infty}(\mathbb{R}_{+};\dot{B}^{-1}_{\infty,\infty})\cap
%%\widetilde{L}^{1}(\mathbb{R}_{+};\dot{B}^{1}_{\infty,\infty})}\leq
%%C_{m}\varepsilon.
%%\end{align}
\medskip

Subsequently, we shall prove \eqref{eq1.6} for all $k,m\geq 0$ by
the standard induction method.  Let $k>0$ be a fixed positive
integer, by using the results of \textbf{Step 1}, we may assume that
\eqref{eq1.6} holds for all $0\leq \ell\leq k-1$ and for all $m\geq
0$, i.e., there hold
\begin{align}\label{eq3.17}
\|t^{\frac{m}{2}+\ell}\left(\partial_{t}^{\ell}\nabla^{m}
u,\partial_{t}^{\ell}\nabla^{m+1}
d\right)\|_{\widetilde{L}^{\infty}(\mathbb{R}_{+};\dot{B}^{-1}_{\infty,\infty})\cap
\widetilde{L}^{1}(\mathbb{R}_{+};\dot{B}^{1}_{\infty,\infty})}\leq
C_{m,k}\varepsilon \quad\forall m\geq 0\text{ and } \ell=1,2,\cdots,
k-1.
\end{align}
%%----------------------------(eq3.17)------------------------------------
In what follows, we shall have completed the proof of Theorem
\ref{thm1.3} if we prove that \eqref{eq3.17} still holds for $k$.
Before going to do it, we notice that by using \eqref{eq3.17} and
the fact that the operator $\nabla^{-1}$ is bounded from
$\dot{B}^{-s-1}_{\infty,\infty}$ to $\dot{B}^{-s}_{\infty,\infty}$
for all $s\geq 0$,  we have
\begin{align}\label{eq3.18}
\|t^{\frac{m}{2}+\ell}\left(\partial_{t}^{\ell}\nabla^{m+1}
u,\partial_{t}^{\ell}\nabla^{m+1}\nabla
d\right)\|_{\widetilde{L}^{\infty}(\mathbb{R}_{+};\dot{B}^{0}_{\infty,\infty})}\leq&
\|t^{\frac{m}{2}+\ell}\left(\partial_{t}^{\ell}\nabla^{m}
u,\partial_{t}^{\ell}\nabla^{m}\nabla
d\right)\|_{\widetilde{L}^{\infty}(\mathbb{R}_{+};\dot{B}^{-1}_{\infty,\infty})}\nonumber\\
\leq& C_{m,k}\varepsilon \quad\forall m\geq 0 \text{ and }
\ell=1,2,\cdots, k-1.
\end{align}
%%----------------------------(eq3.18)------------------------------------
By using the standard interpolation thoery, it follows from
\eqref{eq3.17} that
\begin{align}\label{eq3.19}
\|\left(t^{\frac{m}{2}+\ell}\partial_{t}^{\ell}\nabla^{m}
u,\partial_{t}^{\ell}\nabla^{m}\nabla
d\right)\|_{\widetilde{L}^{2}(\mathbb{R}_{+};\dot{B}^{0}_{\infty,\infty})}\leq
C_{m,k}\varepsilon\quad\forall m\geq 0\text{ and } \ell=1,2,\cdots,
k-1.
\end{align}
%%----------------------------(eq3.19)------------------------------------
We also notice that when we study the linear nonhomogeneous
fractional heat equation
\begin{align*}
\partial_{t}\Psi-\Delta\Psi= F.
\end{align*}
For any $k\geq 1$, by induction we have
\begin{align}\label{eq3.20}
\partial_{t}^{k}\Psi = \Delta^{k} \Psi+\sum_{i=0}^{k-1}
\Delta^{k-1-i}\partial_{t}^{i} F.
\end{align}
%%----------------------------(eq3.20)------------------------------------
 On the other hand, we can
rewrite \eqref{eq1.1} as
\begin{align*}
\partial_{t} u -\Delta u=-\mathbb{P}\nabla\cdot(u\otimes
u+\nabla d\odot \nabla d).
\end{align*}
Hence, by using the inequality \eqref{eq3.20}, and by writing $Y:=
\widetilde{L}^{\infty}(\mathbb{R}_{+};\dot{B}^{-1}_{\infty,\infty})\cap
\widetilde{L}^{1}(\mathbb{R}_{+};\dot{B}^{1}_{\infty,\infty})$, we
have
\begin{align}\label{eq3.21}
\|t^{\frac{m}{2}+k}\partial_{t}^{k}\nabla^{m} u\|_{Y}\leq&
\|t^{\frac{m}{2}+k} \Delta^{k}\nabla^{m}
u\|_{Y}+\sum_{i=0}^{k-1}\|t^{\frac{m}{2}+k}\Delta^{k-1-i}\partial_{t}^{i}\nabla^{m}
(\mathbb{P}\nabla\cdot(u\otimes u)+(\nabla d\odot\nabla d))\|_{Y}\nonumber\\
\leq& \|t^{\frac{m}{2}+k} \nabla^{2 k+m}
u\|_{Y}+\sum_{i=0}^{k-1}\|t^{\frac{m}{2}+k}\partial_{t}^{i}\nabla^{2(k-1-i)+m+1}
(u\otimes u+\nabla d\odot \nabla d)\|_{Y}\nonumber\\
 :=& II_{1}+II_{2}.
\end{align}
%%----------------------------(eq3.21)------------------------------------
Repeating the progress as in derive \eqref{eq3.17} and
\eqref{eq3.20}, it is easy to see
\begin{align*}
 II_{1}\leq C_{m,k}\varepsilon(1+\varepsilon)%\|u_{0}\|_{Q^{\beta,-1}_{\alpha;\infty}}(1+\|u_{0}\|_{Q^{\beta,-1}_{\alpha;\infty}}).
\end{align*}
For the rest term $II_{2}$, by using Lemma \ref{lem2.4}, the
inequalities \eqref{eq3.18} and \eqref{eq3.19}, we have
\begin{align*}
II_{2}\leq&\sum_{i=0}^{k-1}\|t^{\frac{m}{2}+k}\partial_{t}^{i}\nabla^{2(k-1-i)+m+1}
(u\otimes u+\nabla d\odot\nabla d)\|_{\widetilde{L}^{2}(\mathbb{R}_{+};\dot{B}^{0}_{\infty,\infty})}\nonumber\\
\leq&
\sum_{i=0}^{k-1}\sum_{\ell=0}^{2(k-1-i)+m+1}\sum_{\ell_{1}=0}^{\ell}\Big(\left\|
|t^{\frac{\ell}{2}+\ell_{1}}\partial_{t}^{\ell_{1}}\nabla^{\ell}
u|\cdot|t^{\frac{m-\ell}{2}+k-\ell_{1}}\partial_{t}^{i-\ell_{1}}\nabla^{2(k-1-i)+m+1-\ell}
u|\right\|_{\widetilde{L}^{2}(\mathbb{R}_{+};\dot{B}^{0}_{\infty,\infty})}\nonumber\\
&+\left\|
|t^{\frac{\ell}{2}+\ell_{1}}\partial_{t}^{\ell_{1}}\nabla^{\ell}
\nabla
d|\cdot|t^{\frac{m-\ell}{2}+k-\ell_{1}}\partial_{t}^{i-\ell_{1}}\nabla^{2(k-1-i)+m+1-\ell}
\nabla d|\right\|_{\widetilde{L}^{2}(\mathbb{R}_{+};\dot{B}^{0}_{\infty,\infty})}\Big)\nonumber\\
 \leq&
C\sum_{i=0}^{k-1}\sum_{\ell=0}^{2(k-1-i)+m}\sum_{\ell_{1}=0}^{\ell}\left\|
t^{\frac{\ell}{2}+\ell_{1}}\left(\partial_{t}^{\ell_{1}}\nabla^{\ell}
u,\partial_{t}^{\ell_{1}}\nabla^{\ell}
\nabla d\right)\right\|_{\widetilde{L}^{2}(\mathbb{R}_{+};\dot{B}^{0}_{\infty,\infty})}\nonumber\\
\quad &\times
\left\|t^{\frac{m-\ell}{2}+k-\ell_{1}}\left(\partial_{t}^{i-\ell_{1}}\nabla^{2(k-1-i)+m-\ell+1}
u,\partial_{t}^{i-\ell_{1}}\nabla^{2(k-1-i)+m-\ell+1} \nabla
d\right)\right\|_{\widetilde{L}^{\infty}(\mathbb{R}_{+};\dot{B}^{0}_{\infty,\infty})}\nonumber\\
\leq& C_{m,k}\varepsilon^{2}.
\end{align*}
Inserting the estimates of $II_{1}$ and $II_{2}$ into
\eqref{eq3.21}, it follows that
\begin{align*}
\|t^{\frac{m}{2}+k}\partial^{k}_{t}\nabla^{m}
u\|_{\widetilde{L}^{\infty}(\mathbb{R}_{+};\dot{B}^{-1}_{\infty,\infty})\cap
\widetilde{L}^{1}(\mathbb{R}_{+};\dot{B}^{1}_{\infty,\infty})}\leq
C_{m}\varepsilon,
\end{align*}
if we choose $\varepsilon$ small enough.

By rewriting \eqref{eq1.2} as
\begin{align*}
\partial_{t}(\nabla d)-\Delta\nabla d= -\nabla (u\cdot\nabla d+|\nabla
d|^{2}d),
\end{align*}
and by repeating a similar process of the derivations of the
velocity field $u$, we can handle the case for $d$, i.e., there
holds
\begin{align*}
\|t^{\frac{m}{2}+k}\partial^{k}_{t}\nabla^{m} \nabla
d\|_{\widetilde{L}^{\infty}(\mathbb{R}_{+};\dot{B}^{-1}_{\infty,\infty})\cap
\widetilde{L}^{1}(\mathbb{R}_{+};\dot{B}^{1}_{\infty,\infty})}\leq
C_{m}\varepsilon,
\end{align*}
if we choose $\varepsilon$ small enough. Therefore, we conclude that
\eqref{eq1.6} is still holds for $k$. This completes the proof of
Theorem \ref{thm1.3}.
 \hfill$\Box$
%--------------------(proof of thm1.3)------------------------------------

\begin{remark}
It is standard that the condition \eqref{eq1.4} is preserved by the
flow. In fact, by applying the maximum principle to the equations of
$|d|^{2}$, we can easily see that $|d|=1$ under the initial
assumption the $|d_{0}|=1$. We omitted this step due to it is
standard.
\end{remark}

\section{Proof of Theorem \ref{thm1.5}}\label{Proof2}

In this section, following the methods used by Chemin in
\cite{C1,C2}(see also \cite{ZZ}), we give the proof of Theorem
\ref{thm1.5}. We first notice that the proof of the existence of
$\gamma(x,t)$ is exactly as the couterpart in Theorem 3.4 of
\cite{C1} (or Theorem 3.2.1 of \cite{C2}), and we omit the details
here. In what follows, the main issue if to prove \eqref{eq1.7}.
Indeed, for any $x_{1},x_{2}\in\mathbb{R}^{n}$ and $q\in\mathbb{Z}$,
let us decompose $(u,\nabla d)$ in a low and a high frequency part.
This leads to, for all $t\in \mathbb{R}_{+}$, we have
\begin{align*}
&|\gamma(x_{1},t)-\gamma(x_{2},t)|\nonumber\\
\leq&
|x_{1}-x_{2}|+\int_{0}^{t}|
 \left(S_{q}
u(\gamma(x_{1},s),s)-S_{q}u(\gamma(x_{2},s),s), S_{q}
\nabla d(\gamma(x_{1},s),s)-S_{q}\nabla(\gamma(x_{2},s),s)\right)|\text{d}s\nonumber\\
&+2\int_{0}^{t} \sum_{p\geq q}
\|\left(\Delta_{p}u(\cdot,s),\Delta_{p}\nabla d(\cdot,s)\right)\|_{L^{\infty}}\text{d}s\nonumber\\
\leq& |x_{1}-x_{2}|+\int_{0}^{t} \|\left(\nabla S_{q}
u(\cdot,s),\nabla S_{q}\nabla d(\cdot,s)\right)\|_{L^{\infty}}|\gamma(x_{1},s)-\gamma(x_{2},s)|\text{d}s\nonumber\\
&+2\sum_{p\geq q} 2^{-p}\cdot\sup_{p\geq q}\left\{\int_{0}^{t} 2^{p}
\|\left(\Delta_{p}u(\cdot,s),\Delta_{p}\nabla d(\cdot,s)\right)\|_{L^{\infty}}\text{d}s\right\}\nonumber\\
\leq& |x_{1}-x_{2}|+\int_{0}^{t} \|\left(\nabla S_{q}
u(\cdot,s),\nabla S_{q}
\nabla d(\cdot,s)\right)\|_{L^{\infty}}|\gamma(x_{1},s)-\gamma(x_{2},s)|\text{d}s\nonumber\\
&+2^{2-q} \|(u,\nabla
d)\|_{\widetilde{L}^{\infty}(\mathbb{R}_{+};\dot{B}^{1}_{\infty,\infty})}.
\end{align*}
Let $\rho(t):\triangleq |\gamma(x_{1},t)-\gamma(x_{2},t)|$ and
\begin{align*}
D_{q}:\triangleq |x_{1}-x_{2}|+2^{2-q} \|(u,\nabla
d)\|_{\widetilde{L}^{\infty}(\mathbb{R}_{+};\dot{B}^{1}_{\infty,\infty})}\!+\!\int_{0}^{t}\!
\|\left(\nabla S_{q} u(\cdot,s),\nabla S_{q}\nabla
d(\cdot,s)\right)\|_{L^{\infty}}|\gamma(x_{1},s)-\gamma(x_{2},s)|\text{d}s.
\end{align*}
Then we have
\begin{align*}
\rho(t) \leq D_{q}(t)\quad\text{ for  all  }q\in\mathbb{Z},
\end{align*}
and
\begin{align*}
D_{q}(t)\leq  |x_{1}-x_{2}|+2^{2-q} \|(u,\nabla
d)\|_{\widetilde{L}^{\infty}(\mathbb{R}_{+};\dot{B}^{1}_{\infty,\infty})}+\int_{0}^{t}
\|\left(\nabla S_{q} u(\cdot,s),\nabla S_{q}\nabla
d(\cdot,s)\right)\|_{L^{\infty}}D_{q}(s)\text{d}s.
\end{align*}
The Gronwall inequality implies that, for any $t> 0$,
\begin{align}\label{eq4.1}
D_{q}(t)\leq&  ( |x_{1}-x_{2}|+2^{2-q} \|(u,\nabla
d)\|_{\widetilde{L}^{\infty}(\mathbb{R}_{+};\dot{B}^{1}_{\infty,\infty})})\exp\left\{\int_{0}^{t}
\|\left(\nabla S_{q} u(\cdot,s),\nabla S_{q}\nabla
d(\cdot,s)\right)\|_{L^{\infty}}\text{d}s\right\}
\nonumber\\
\leq &  ( |x_{1}-x_{2}|+C 2^{2-q}
\varepsilon)\exp\left\{\int_{0}^{t} \|\nabla S_{q}
u(\cdot,s)\|_{L^{\infty}}\text{d}s\right\},
\end{align}
%%----------------------------(eq4.1)------------------------------------
where we have used \eqref{eq1.6} with $k=m=0$ in the last inequality
above. Notice that the above inequality holds for all
$q\in\mathbb{Z}$, and by selecting $q\geq 1$, we have
\begin{align*}
&\int_{0}^{t} \|\left(\nabla S_{q} u(\cdot,s),\nabla S_{q}\nabla
d(\cdot,s)\right)\|_{L^{\infty}}\text{d}s\nonumber\\
\leq& \int_{0}^{t} \sum_{p\leq 0} 2^{p} \|\left(\Delta_{p}
u(\cdot,s),\Delta_{p}\nabla
d(\cdot,s)\right)\|_{L^{\infty}}\text{d}s+\sum_{p=0}^{q}
\int_{0}^{t}
2^{p}\|(u(\cdot,s),\nabla d(\cdot,s))\|_{L^{\infty}}\text{d}s\nonumber\\
\leq& \int_{0}^{t} \sum_{p\leq 0} 2^{2p} \cdot 2^{-p}
\|\left(\Delta_{p} u(\cdot,s),\Delta_{p}\nabla
d(\cdot,s)\right)\|_{L^{\infty}}\text{d}s+q
\|(u,\nabla d)\|_{\widetilde{L}^{1}_{t}(\dot{B}^{1}_{\infty,\infty})}\nonumber\\
\leq& C t\|(u,\nabla
d)\|_{\widetilde{L}^{\infty}(\mathbb{R}_{+};\dot{B}^{-1}_{\infty,\infty})}+q
\|(u,\nabla d)\|_{\widetilde{L}^{1}(\mathbb{R}_{+}\dot{B}^{1}_{\infty,\infty})}\nonumber\\
\leq &C(t+q)\|(u,\nabla
d)\|_{\widetilde{L}^{\infty}(\mathbb{R}_{+};\dot{B}^{-1}_{\infty,\infty})\cap
\widetilde{L}^{1}(\mathbb{R}_{+};\dot{B}^{1}_{\infty,\infty})}\leq C
\varepsilon (t+q).
\end{align*}
The above inequality along with \eqref{eq4.1} implies that
\begin{align*}
D_{q}(t)\leq ( |x_{1}-x_{2}|+C 2^{2-q} \varepsilon)\exp\left\{ C
\varepsilon (t+q)\right\}.
\end{align*}
By choosing $2^{q}\equiv |x_{1}-x_{2}|^{-1}$ in the above
inequality, we infer that \eqref{eq1.7} holds, and the Theorem
\ref{thm1.5} is proved.
 \hfill$\Box$
\\
%\\
%\textbf{Acknowledgments}

\end{document}